\documentclass[reqno,12pt]{amsart}
\usepackage{amsfonts}
\usepackage{upgreek}
\usepackage{bbm}
\usepackage{} %leqno is the option to put formula numbers on the left side
\numberwithin{equation}{section}
\usepackage{indentfirst}
 \usepackage{color}
\usepackage{amssymb}
\usepackage{mathrsfs}
\usepackage{xy}
\xyoption{all}
\def\Ext{\mbox{\rm Ext}\,} \def\Hom{\mbox{\rm Hom}} \def\dim{\mbox{\rm dim}\,} \def\Iso{\mbox{\rm Iso}\,}
\def\lr#1{\langle #1\rangle}    \def\top{\mbox{\rm top}\,}\def\soc{\mbox{\rm soc}\,}
\def\Ker{\mbox{\rm Ker}\,}    \def\Coker{\mbox{\rm Coker}\,}
\def\End{\mbox{\rm End}\,}

\def\Dim{\mbox{\rm \textbf{dim}}\,}\def\A{\mathcal{A}\,} \def\D{\mathcal{D}}
\def\P{\mathcal{P}\,}\def\I{\mathcal{I}\,}

\def\ZZ{\mathbb Z}

\def\x{{\bf x}}

\theoremstyle{plain} %text of this environment is typesetted in italics
\newtheorem{theorem}{\bf Theorem}[section]
\newtheorem{lemma}[theorem]{\bf Lemma}
\newtheorem{corollary}[theorem]{\bf Corollary}
\newtheorem{proposition}[theorem]{\bf Proposition}

\theoremstyle{definition} %text of this environment is typesetted in roman letters
\newtheorem{definition}[theorem]{\bf Definition}
\newtheorem{remark}[theorem]{\bf Remark}
\newtheorem{example}[theorem]{\bf Example}

\newcommand{\bt}{\begin{theorem}}
\newcommand{\et}{\end{theorem}}
\newcommand{\bl}{\begin{lemma}}
\newcommand{\el}{\end{lemma}}
\newcommand{\bd}{\begin{definition}}
\newcommand{\ed}{\end{definition}}
\newcommand{\bc}{\begin{corollary}}
\newcommand{\ec}{\end{corollary}}
\newcommand{\bp}{\begin{proof}}
\newcommand{\ep}{\end{proof}}
\newcommand{\bx}{\begin{example}}
\newcommand{\ex}{\end{example}}
\newcommand{\br}{\begin{remark}}
\newcommand{\er}{\end{remark}}
\newcommand{\be}{\begin{equation}}
\newcommand{\ee}{\end{equation}}
\newcommand{\ba}{\begin{align}}
\newcommand{\ea}{\end{align}}
\newcommand{\bn}{\begin{enumerate}}
\newcommand{\en}{\end{enumerate}}
\newcommand{\bcs}{\begin{cases}}
\newcommand{\ecs}{\end{cases}}

\makeatletter
\renewcommand{\section}{\@startsection{section}{1}{0mm}
  {-\baselineskip}{0.5\baselineskip}{\bf\leftline}}
\makeatother

\begin{document}

\title[The cluster multiplication theorem for quantum cluster algebras]{The cluster multiplication theorem for\\ acyclic quantum cluster algebras} %title of paper and the running head option

\author{Xueqing Chen, Ming Ding and Haicheng Zhang$^*$}
\address{Department of Mathematics, University of Wisconsin-Whitewater\\
800 W. Main Street, Whitewater, WI.53190. USA}
\email{chenx@uww.edu (X.Chen)}
\address{School of Mathematics and Information Science\\
Guangzhou University, Guangzhou 510006, P.~R.~China}
\email{m-ding04@mails.tsinghua.edu.cn (M. Ding)}
%\address{Department of Mathematical Sciences\\
%Tsinghua University\\
%Beijing 100084, P.~R.~China}
%\email{fanxu@mail.tsinghua.edu.cn(F. Xu)}
\address{Institute of Mathematics, School of Mathematical Sciences, Nanjing Normal University,
Nanjing 210023, P.~R.~China}
\email{zhanghc@njnu.edu.cn (H. Zhang)}

%%%%%%%%%%%%%%% footnote %%%%%%%%%%%%%%%%
%In case \subjclass[2010] command is not effective
%(or the version of amsart.cls is old), write as follows:
%\renewcommand{\thefootnote}{\fnsymbol{footnote}}
%\footnote[0]{2010\textit{ Mathematics Subject Classification}.
%Primary 00; Secondary 00.}
%
\thanks{$*$~Corresponding author.}
%%%%%%%%%%%% Authors addresses %%%%%%%%%%%%%

%%%%%%%%%%%%%%%%%%%%%%%%%%%%%%%%%%%%%%%%%

\begin{abstract}
Let $Q$ be a finite acyclic valued quiver. We give the cluster multiplication formulas in the quantum cluster algebra of $Q$ with arbitrary coefficients, by applying certain quotients of derived Hall subalgebras of $Q$.
These formulas can be viewed as the quantum version of the cluster multiplication theorem in the classical cluster algebra proved by Caldero-Keller for finite type, Hubery for affine type and Xiao-Xu for acyclic quivers.

\end{abstract}

\maketitle

\section{Introduction}

In order to study total positivity in algebraic groups and canonical bases in quantum groups,  Fomin and Zelevinsky~\cite{FZ} invented the cluster algebra which is a commutative algebra generated by a family of generators called cluster variables. The set of all cluster variables is constructed recursively via mutations from initial cluster variables.
Cluster categories are certain quotients of derived categories of representations of finite dimensional algebras, which were introduced in~\cite{BMRRT} to ``categorify'' cluster algebras.
%The indecomposable objects of the cluster category are exactly in bijection with the almost positive roots of the associated root system, and there is a bijection between the cluster-tilting objects of cluster category and the clusters of the cluster algebra.(ÕâÀïµÄ±íŽïÏÂÃæÓÖ³öÏÖÁË)
Cluster categories have led to new development in the theory of the (dual) canonical bases and provided an insight into cluster algebras.

The connections between the cluster algebras and the cluster categories are explicitly characterized by the Caldero-Chapoton map in \cite{CC} and the
Caldero-Keller multiplication theorem in \cite{CK2005,CK2}. The
Caldero-Chapoton map associates the objects in the cluster
categories to some Laurent polynomials. Indeed, using the Caldero-Chapoton map, for acyclic cases, Caldero
and Keller established a one-to-one correspondence between the cluster tilting objects and the indecomposable rigid objects of the cluster category and the clusters and the cluster variables of the associated cluster algebra, respectively.

%%%%%%%%%%%%%%%%%%%%%%%%%%%%%%%%%%%%%%%%%%%%%%%%%%%%%%%%%%%

Let $Q$ be a Dynkin quiver and $\mathcal{C}_Q$ the cluster category of $Q$. The  Caldero-Chapoton character $X_M$ only depends on the isomorphism class of the object
$M$ in  $\mathcal{C}_Q$.  It was proved in \cite{CC} that
\begin{equation}\label{hall}
X_M X_N =X_{M\oplus N}
\end{equation}
for any objects $M$ and $N$ in $\mathcal{C}_Q$.
Caldero and Keller~\cite{CK2005} proved the following cluster multiplication formula:
\begin{equation}\label{general}
{\chi (\mathbb{P} \operatorname{Ext}^1(M,N))} X_M X_N =\sum_E \displaystyle { \big( \chi (\mathbb{P} \operatorname{Ext}^1(M,N)_E) +\chi (\mathbb{P} \operatorname{Ext}^1(N,M)_E) \big)} X_E
\end{equation}
for any objects $M, N$ in $\mathcal{C}_Q$ such that $\operatorname{Ext}^1(M,N) \neq 0$,
where $\chi$ is the Euler-Poincar\'e characteristic of \'etale cohomology with proper support  and  the sum is taken over all isomorphism classes of objects $E$ in $\mathcal{C}_Q$. Note that the coefficient of $X_{M\oplus N}$
on the righthand side is $0$. This formula is called the {\em cluster multiplication theorem} for finite type.

Let $Q$ be an acyclic quiver. For any pair of indecomposable objects $M$  and $N$ of the cluster category $\mathcal{C}_Q$ such that $\operatorname{Ext}^1(M,N)$ is one-dimensional,  Caldero and Keller \cite{CK2} showed that
\begin{equation}\label{one-dim}
X_M X_N= X_E+X_{E'},
\end{equation}
where $E$ and $E'$ are the unique objects (up to morphism) such that there exist non-split triangles
$N\rightarrow E\rightarrow M\rightarrow N[1]$ and $M\rightarrow E'\rightarrow N\rightarrow M[1]$.

%Xiao-Xu~\cite{XX} proved Equation (\ref{general}) for acyclic quivers.

%%%%%%%%%%%%%%%%%%%%%%%%%%%%%%%%%%%%%%%%%%%%%%%%%%%%%%%%%%%

%For simply laced Dynkin quivers, Caldero and Keller constructed a
%cluster multiplication formula between two
%cluster characters in \cite{CK2005}, which is similar to the multiplication in a dual Hall algebra.
%It combines homological and geometric properties of cluster
%categories and combinatorial properties of cluster algebras.

The cluster multiplication theorem in (\ref{general}) was generalized to affine type by Hubery in \cite{Hubery1} and to any type by Xiao-Xu in \cite{XX, Xu}. Palu \cite{Palu2} extended the equation (\ref{general}) to $2$-Calabi-Yau triangulated categories with cluster tilting objects. Fu and Keller \cite{FuK} further generalized Palu's results to $2$-Calabi-Yau Frobenius exact categories and certain subcategories of $2$-Calabi-Yau triangulated categories.  In the cluster theory, the Caldero-Chapoton map and the cluster multiplication theorem play a
very important role in proving some structural results such as bases with good properties,
positivity conjecture, denominator conjecture and so on~(cf. \cite{CK2005,DXX}).

The quantum cluster algebras, as the quantum deformations of cluster algebras, were defined in~\cite{BZ05} to serve as an algebraic framework for the study of dual canonical bases in coordinate rings and their q-deformations.
Rupel \cite{Rupel1} defined a quantum analogue of the Caldero-Chapoton map for the category of representations of an acyclic valued quiver over a finite field.
The quantum version of the equation (\ref{one-dim}) was proved by Rupel in \cite{Rupel1} for indecomposable rigid objects for all finite type valued quivers and rank 2 valued quivers,
by Qin ~\cite{Qin} for indecomposable rigid objects for acyclic quivers, by Ding-Xu~\cite{DX} for more objects, and then was generalized by Rupel~\cite{Rupel2} for acyclic valued quivers.
In \cite{DX, Fei, BR, DSC, FPZ}, the quantum version of the equation~(\ref{hall}) was confirmed for the quantum cluster algebra of an acyclic quiver, see the equations (\ref{qre3}), (\ref{qre6}) and (\ref{qre7}) in Corollary~\ref{Hallcfgs}.

The aim of this paper is to provide a quantum version of the cluster multiplication theorem in (\ref{general}) for acyclic quantum cluster algebras.
In particular, we also obtain a quantum version of the cluster multiplication formulas in (\ref{one-dim}).
To the best of our knowledge, in the literature, the quantum cluster multiplication theorem has only been deduced for the rank two affine cases (cf. \cite{dx1,bcdx}).

The cluster multiplication theorem showed the similarity between the multiplication in a cluster algebra and that in a dual Hall algebra, which motivated a natural idea to construct a framework to explicitly relate the dual Hall algebras with the quantum cluster algebras.
The Hall algebra of a finite dimensional algebra $A$ over a finite field was introduced by Ringel~\cite{R90,R90a} in 1990.
%Ringel \cite{R90,R90a} proved that the Hall algebra of a representation-finite hereditary algebra provides a realization of the positive part of the corresponding quantum group. Ringel's approach establishes a relation between the representation theory of algebras and Lie theory, and provides an algebraic framework for studying the Lie theory resulting from Hall algebras associated to various exact categories.
To\"en \cite{Toen2006} generalized Ringel's construction to define the derived Hall algebra for a DG-enhanced triangulated category satisfying certain finiteness conditions. Later on, for a triangulated category satisfying the left homological finiteness condition, Xiao and Xu \cite{XiaoXu} showed that To\"en's construction still provides an associative unital algebra.
%It was expected but so far not successful to realize the entire quantum group via the derived Hall algebra of a triangulated category. In 2013,
%Bridgeland~\cite{Bri13} provided a realization of the
%whole quantum group via the Hall algebra of 2-cyclic complexes of projective modules over a hereditary algebra.

%Note that this theorem provides a `Hall algebra type' multiplication formula for the cluster algebra since in a dual Hall algebra of $Q$, we have  $$ [M] \star [M] =\sum_{[N]}  \displaystyle\frac{|\operatorname{Ext}^1 (L,M)_N|}{|%\operatorname{Ext}^1 (L,M)|} [N].$$

Given a finite acyclic quiver $Q$, let $\mathcal{AH}_{q}(Q)$ be the subalgebra of a certain skew-field of
fractions generated by quantum cluster characters (see Section~\ref{q-cluster} for more details). Then, there
exists an algebra homomorphism from the dual Hall algebra associated to the representation
category of $Q$ to $\mathcal{AH}_{q}(Q)$ (cf. \cite{CDX}). However, this homomorphism may not be surjective, in particular, there might be no preimages of  the initial quantum cluster variables.
In order to overcome this shortcoming, in \cite{DXZ}, the authors considered the morphism category $\mathcal{C}_2(\mathcal{P})$ of projective representations of $Q$, which has indecomposable objects indexed by the isomorphism classes of indecomposable objects in the cluster category $\mathcal{C}_Q$.  They constructed the localized Hall algebra $\mathcal{MH}_q(Q)$ associated to $\mathcal{C}_2(\mathcal{P})$ which contains the dual Hall algebra of $Q$ as a subalgebra, and obtained a surjective algebra homomorphism from  $\mathcal{MH}_\Lambda (Q)$ (a certain twisted version of $\mathcal{MH}_{q}(Q)$) to $\mathcal{AH}_{q}(Q)$, and then realized the quantum cluster algebra as a sub-quotient algebra of this Hall algebra. In fact, the algebra $\mathcal{MH}_{q}(Q)$ is isomorphic to a subalgebra of the extended dual derived Hall algebra of $Q$. Recently, Fu, Peng and Zhang~\cite{FPZ} provided a bialgebra structure and an integration homomorphism on  $\mathcal{MH}_{q}(Q)$. Then they recovered the surjective homomorphism
defined in~\cite{DXZ}, and also recovered the quantum Caldero-Chapoton map, as well as some cluster multiplication formulas.

In order  to get the cluster multiplication theorem for the quantum cluster algebra associated to an acyclic valued quiver $Q$,
we will use the derived Hall algebra instead of  $\mathcal{MH}_{q}(Q)$. Following~\cite{FPZ}, we similarly define the comultiplications and integration maps on the subalgebras
of derived Hall algebras to obtain the formulas of (generalised) quantum cluster characters. Then we introduce two certain quotients of the derived Hall subalgebras and use these quotient algebras to construct two multiplication formulas between $u_M$ and $u_{P[1]}$, and between $u_M$ and $u_N$, respectively. For the first multiplication formula, we immediately get the corresponding cluster multiplication formula by establishing an algebra homomorphism from the defined algebra to the quantum torus $\mathcal{T}_\Lambda$. The second multiplication formula inspired us to directly prove the cluster multiplication formula between
$X_M$ and $X_N$.  These two cluster multiplication formulas (Theorem~\ref{dyggs} and Theorem \ref{ddlz}) can be viewed as the quantum version of the cluster multiplication
theorem in (1.2) for acyclic quantum cluster algebras. Recently, Chen, Xiao and Xu \cite{CXX} proved a general multiplication formula between weighted quantum cluster characters. It is worthy of a further investigation into the relations between their formula and ours.

%we directly prove the corresponding formula in $\mathcal{T}_\Lambda$ via the quantum cluster algebra approach. The direct proof removes the condition that the one-sided extension vanishes needed in obtaining the second %multiplication formula in the Hall algebra.

The paper is organized as follows: In Section 2 we recall the definition of derived Hall algebra and introduce its three
special subalgebras in order to relate with quantum cluster algebras.  The comultiplications and integration maps on these subalgebras are provided in Section 3. In Section 4 the formulas of quantum cluster characters are constructed by applying the comultiplications and integration maps. We introduce the first quotient algebra to prove the first multiplication formula and the cluster multiplication theorem for the quantum cluster characters $X_M$ and $X_{P[1]}$ in Section 5, and introduce the second one to prove the second  multiplication formula in Section 6.
Finally, in Section 7, we directly prove the cluster multiplication theorem for the quantum cluster characters $X_M$ and $X_{N}$.

Let us fix some notations used throughout the paper. For a finite set $S$, we denote by $|S|$ its cardinality. Let $k=\mathbb{F}_q$ be a finite field with $q$ elements, and set $v=\sqrt{q}$. Let $\ZZ[v,v^{-1}]$ be the ring of integral Laurent polynomials. Let $A$ be a finite dimensional hereditary $k$-algebra, and denote by $\A$ the category of finite dimensional left $A$-modules; let $\P=\P_{\A}$ and $\I=\I_{\A}$ be the subcategories of $\A$ consisting of projective objects and injective objects, respectively. For a module $M\in\A$, we use $\Dim M$ or $\hat{M}$ to denote its dimension vector. For an abelian or triangulated category $\mathcal {E}$, the Grothendieck group of $\mathcal {E}$ and the set of isomorphism classes $[X]$ of objects in $\mathcal {E}$ are denoted by $K(\mathcal {E})$ and $\Iso(\mathcal {E})$, respectively. For each object $M$ in $\mathcal {E}$, the image of $M$ in $K(\mathcal {E})$ is denoted by $\hat{M}$, which we identify with $\Dim M$ if $\mathcal {E}=\A$.
We always assume that all the vectors are column vectors, and all tensor products are taken over $\ZZ[v,v^{-1}]$.

\section{Derived Hall algebras}
In this section, we recall the definition of the derived Hall algebra, and introduce three of its subalgebras.

The derived Hall algebra of the bounded derived category $D^b(\A)$ of $\A$ was introduced by To\"{e}n in \cite{Toen2006} (see also \cite{XiaoXu}). By definition, the (Drinfeld dual) {\em derived Hall algebra} $\mathcal {D}\mathcal {H}(\A)$ is the free $\ZZ[v,v^{-1}]$-module with the basis $\{u_{X_\bullet}~|~X_\bullet\in \Iso(D^b(\A))\}$ and the multiplication defined by
\begin{equation}
u_{X_\bullet}\diamond u_{Y_\bullet}=\sum\limits_{[{Z_\bullet}]}\frac{|\Ext^1_{D^b(\A)}({X_\bullet},{Y_\bullet})_{Z_\bullet}|}{\prod\limits_{i\geq0}|\Hom_{D^b(\A)}({X_\bullet}[i],{Y_\bullet})|^{(-1)^i}} u_{Z_\bullet},
\end{equation}
where $\Ext^1_{D^b(\A)}({X_\bullet},{Y_\bullet})_{Z_\bullet}$ is defined to be $\Hom_{D^b(\A)}({X_\bullet},{Y_\bullet}[1])_{{Z_\bullet}[1]}$, which denotes the subset of $\Hom_{D^b(\A)}({X_\bullet},{Y_\bullet}[1])$ consisting of morphisms $f:{X_\bullet}\rightarrow {Y_\bullet}[1]$ whose cone is isomorphic to ${Z_\bullet}[1]$.

For any ${X_\bullet},{Y_\bullet}\in D^b(\A)$, define
\begin{equation*}
\lr{{X_\bullet},{Y_\bullet}}:=\sum\limits_{i\in\mathbb{Z}}(-1)^i\dim_k\Hom_{D^b(\A)}({X_\bullet},{Y_\bullet}[i]).
\end{equation*}
The sum above descends to give a bilinear form on the Grothendieck group of $D^b(\A)$. Moreover, this bilinear form coincides with the Euler form of $K(\A)$ over the objects in $\A$. In particular, for any $M,N\in\A$ and $i,j\in\mathbb{Z}$, we have that $\lr{M[i],N[j]}=(-1)^{i-j}\lr{M,N}$.

In what follows, given four objects $M,N,X,Y\in\A$, we set
$$_X\Hom_{\A}(M,N)_Y:=\{f: M\rightarrow N~|~\Ker f\cong X~\text{and}~\Coker f\cong Y\}.$$

Let us twist the multiplication in $\mathcal {D}\mathcal {H}(\A)$ as follows:
\begin{equation}u_{X_\bullet}\ast u_{Y_\bullet}=q^{\lr{{X_\bullet},{Y_\bullet}}} u_{X_\bullet}\diamond u_{Y_\bullet}\end{equation}
for any ${X_\bullet},{Y_\bullet}\in D^b(\A)$.
The \emph{twisted derived Hall algebra} $\mathcal {D}\mathcal {H}_{q}(\A)$ is the same module as $\mathcal {D}\mathcal {H}(\A)$, but with the twisted multiplication.
\begin{proposition}{\rm(\cite{Toen2006})}\label{twistderived}
The twisted derived Hall algebra $\mathcal {D}\mathcal {H}_q(\A)$ is an associative unital algebra generated by the elements in $\{u_{M[i]}~|~M\in\Iso(\A),~i\in \mathbb{Z}\}$ and the following relations
\begin{flalign}
&u_{M[i]}\ast u_{N[i]}=q^{\lr{M,N}}\sum\limits_{[L]}{\frac{{|\Ext_\mathcal{A}^1{{(M,N)}_L}|}}{{|\Hom_\mathcal{A}(M,N)|}}}u_{L[i]};\\
&u_{M[i]}\ast u_{N[i+1]}=u_{M[i]\oplus N[i+1]};\\
&u_{M[i+1]}\ast u_{N[i]}=q^{-\lr{M,N}}\sum\limits_{[X],[Y]}|{}_X\Hom_{\A}(M,N)_Y| u_{Y[i]}\ast u_{X[i+1]};\\
&u_{M[i]}\ast u_{N[j]}=q^{(-1)^{i-j}\lr{M,N}} u_{N[j]}\ast u_{M[i]}, \quad i-j>1.
\end{flalign}
\end{proposition}

In what follows, we will consider the following three subalgebras of the twisted derived Hall algebra $\mathcal {D}\mathcal {H}_q(\A)$.
Let $C_{\A}^1$ and $C_{\A}^2$ be the subcategories of $D^b(\A)$ consisting of objects $M\oplus P[1]$ and $M\oplus I[-1]$ with $M\in\A,P\in\P,I\in\I$, respectively.
Let $C_{\A}^{e}$ be the subcategory of $D^b(\A)$ consisting of objects $I[-1]\oplus M\oplus P[1]$ with $I\in\I,M\in\A,P\in\P$. Since $\A$ is hereditary, it is easy to see that these subcategories are closed under extensions in $D^b(\A)$. Hence,
the submodules $\mathcal {D}\mathcal {H}_q^{c_1}(\A)$, $\mathcal {D}\mathcal {H}_q^{c_2}(\A)$ and $\mathcal {D}\mathcal {H}_q^{ec}(\A)$ of $\mathcal {D}\mathcal {H}_q(\A)$ spanned by all elements $u_{X_\bullet}$ with $X_\bullet$ being in $C_{\A}^1$, $C_{\A}^2$ and $C_{\A}^{e}$, respectively, are three subalgebras of $\mathcal {D}\mathcal {H}_q(\A)$.

Using Proposition \ref{twistderived}, we obtain the following characterizations on these subalgebras.
%\begin{proposition}\label{subalgebra1}
%The subalgebra $\mathcal {D}\mathcal {H}_q^{c_1}(\A)$ is generated by the elements $\{u_{M}, u_{P[1]}~|~M\in\A, P\in\P\}$, and the following relations
%\begin{flalign}
%u_{P[1]}\ast u_{Q[1]}=
%u_{(P\oplus Q)[1]}
%=u_{Q[1]}\ast u_{P[1]};\end{flalign}
%\begin{flalign}&u_{M}\ast u_{N}=q^{\lr{M,N}}\sum_{[L]}\frac{|\mathrm{Ext}_{\A}^{1}(M,N)_{L}|}{|\mathrm{Hom}_{\A}(M,N)|}u_L;\\
%&u_{M}\ast u_{P[1]}=u_{M\oplus
%P[1]};\\
%&u_{P[1]}\ast u_{M}=q^{-\lr{P,M}}
%\sum\limits_{[F],[P']}|{}_{P'}\Hom_{\A}(P,M)_F|u_{F\oplus P'[1]};\end{flalign}
%for any $M,N\in\A$ and $P,Q\in\P$.
%\end{proposition}
%\begin{proposition}\label{subalgebra2}
%The subalgebra $\mathcal {D}\mathcal {H}_q^{c_2}(\A)$ is generated by the elements $\{u_{I[-1]}, u_{M}~|~I\in\I, M\in\A\}$, and the following relations
%\begin{flalign}
%&u_{I[-1]}\ast u_{J[-1]}=
%u_{(I\oplus J)[-1]}
%=u_{J[-1]}\ast u_{I[-1]};\\&
%u_{M}\ast u_{N}=q^{\lr{M,N}}\sum_{[L]}\frac{|\mathrm{Ext}_{\A}^{1}(M,N)_{L}|}{|\mathrm{Hom}_{\A}(M,N)|}u_L;\\
%&u_{I[-1]}\ast u_{M}=u_{M\oplus
%I[-1]};\\
%&u_{M}\ast u_{I[-1]}=q^{-\lr{M,I}}
%\sum\limits_{[G],[I']}|{}_{G}\Hom_{\A}(M,I)_{I'}|u_{G\oplus I'[-1]};\end{flalign}
%for any $M,N\in\A$ and $I,J\in\I$.
%\end{proposition}
\begin{proposition}\label{esubalgebra}
The subalgebra $\mathcal {D}\mathcal {H}_q^{ec}(\A)$ is generated by the elements $\{u_{P[1]},u_{M},u_{I[-1]}~|\\P\in\P,M\in\A,I\in\I\}$, and the following relations
\begin{flalign}
&u_{P[1]}\ast u_{Q[1]}=
u_{(P\oplus Q)[1]}
=u_{Q[1]}\ast u_{P[1]};\label{ecr5}\\
&u_{M}\ast u_{P[1]}=u_{M\oplus
P[1]};\label{ecr6}\\
&u_{P[1]}\ast u_{M}=q^{-\lr{P,M}}
\sum\limits_{[F],[P']}|{}_{P'}\Hom_{\A}(P,M)_F|u_{F\oplus P'[1]};\label{ecr8}\\
&u_{M}\ast u_{N}=q^{\lr{M,N}}\sum_{[L]}\frac{|\mathrm{Ext}_{\A}^{1}(M,N)_{L}|}{|\mathrm{Hom}_{\A}(M,N)|}u_L;\label{ecr4}\\
&u_{I[-1]}\ast u_{J[-1]}=
u_{(I\oplus J)[-1]}
=u_{J[-1]}\ast u_{I[-1]};\label{ecr1}\\
&u_{I[-1]}\ast u_{M}=u_{M\oplus
I[-1]};\label{ecr2}\\
&u_{M}\ast u_{I[-1]}=q^{-\lr{M,I}}
\sum\limits_{[G],[I']}|{}_{G}\Hom_{\A}(M,I)_{I'}|u_{G\oplus I'[-1]};\label{ecr3}\\
&u_{I[-1]}\ast u_{P[1]}=q^{-\lr{P,I}}u_{P[1]}\ast u_{I[-1]}=u_{I[-1]\oplus P[1]};\label{ecr7}\end{flalign}
for any $P,Q\in\P$, $M,N\in\A$ and $I,J\in\I$.
\end{proposition}

The algebras $\mathcal {D}\mathcal {H}_q^{c_1}(\A)$ and $\mathcal {D}\mathcal {H}_q^{c_2}(\A)$ are the subalgebras of $\mathcal {D}\mathcal {H}_q^{ec}(\A)$, which are subject to the relations (\ref{ecr5}-\ref{ecr4}) and (\ref{ecr4}-\ref{ecr3}), respectively.

\section{Comultiplications and integration maps}
In this section, we give the comultiplications and integration maps on the subalgebras of the derived Hall algebra.

Given objects $L,M,N\in\A$, the {\em Ringel-Hall number} $F_{MN}^L$ is defined as the number of subobjects $U$ of $L$ such that $U\cong N$ and $L/U\cong M$.

We give a comultiplication
$$\Delta:\mathcal {D}\mathcal {H}_q^{ec}(\A)\longrightarrow \mathcal {D}\mathcal {H}_q^{ec}(\A)\otimes \mathcal {D}\mathcal {H}_q^{ec}(\A)$$
defined by
\begin{equation}\label{ecyucheng}\Delta(u_{I[-1]\oplus L\oplus P[1] }):=\sum\limits_{[M],[N]}
q^{\lr{\hat{M}-\hat{I},\hat{N}-\hat{P}}}F_{MN}^L(u_{M\oplus I[-1]}\otimes u_{N\oplus P[1]})
\end{equation} for any $L\in\A$, $I\in\I$ and $P\in\P$.
Define the multiplication $\ast$ on $\mathcal {D}\mathcal {H}_q^{ec}(\A)\otimes\mathcal {D}\mathcal {H}_q^{ec}(\A)$ by
\begin{equation}\label{ectwist}
\begin{split}
&(u_{I[-1]\oplus M\oplus P[1] }\otimes u_{J[-1]\oplus N\oplus Q[1] })\ast(u_{I'[-1]\oplus M'\oplus P'[1] }\otimes u_{J'[-1]\oplus N'\oplus Q'[1] }):=\\
&q^{x_0}
(u_{I[-1]\oplus M\oplus P[1] }\ast u_{I'[-1]\oplus M'\oplus P'[1] }\otimes u_{J[-1]\oplus N\oplus Q[1] }\ast u_{J'[-1]\oplus N'\oplus Q'[1] }),
\end{split}
\end{equation} where $x_0=(\hat{N}-\hat{J}-\hat{Q},\hat{M'}-\hat{I'}-\hat{P'})+\lr{\hat{M}-\hat{I}-\hat{P},\hat{N'}-\hat{J'}-\hat{Q'}}$, for any $M, N, M', N'\in\A$, $I, J, I', J'\in\I$ and $P, Q, P', Q'\in\P$.
\begin{proposition}\label{ecp}
The map $\Delta:(\mathcal {D}\mathcal {H}_q^{ec}(\A),\ast)\longrightarrow (\mathcal {D}\mathcal {H}_q^{ec}(\A)\otimes\mathcal {D}\mathcal {H}_q^{ec}(\A),\ast)$ is a homomorphism of algebras.
\end{proposition}
\begin{proof} By \cite[Proposition 8.5]{FPZ}, the restriction of $\Delta$ to $\mathcal {D}\mathcal {H}_q^{c_1}(\A)$ is a homomorphism of algebras. Dually, its restriction to $\mathcal {D}\mathcal {H}_q^{c_2}(\A)$ is also a homomorphism of algebras. Thus, it suffices to check that $\Delta$ preserves the relation (\ref{ecr7}). This is immediate.
\end{proof}
As we have mentioned in the proof of Proposition \ref{ecp}, we have the following
\begin{corollary}{\rm(\cite[Proposition 8.5]{FPZ})}\label{c1p}
The map $\Delta:(\mathcal {D}\mathcal {H}_q^{c_1}(\A),\ast)\longrightarrow (\mathcal {D}\mathcal {H}_q^{c_1}(\A)\otimes\mathcal {D}\mathcal {H}_q^{c_1}(\A),\ast)$ is a homomorphism of algebras.
\end{corollary}
\begin{corollary}\label{c2p}
The map $\Delta:(\mathcal {D}\mathcal {H}_q^{c_2}(\A),\ast)\longrightarrow (\mathcal {D}\mathcal {H}_q^{c_2}(\A)\otimes\mathcal {D}\mathcal {H}_q^{c_2}(\A),\ast)$ is a homomorphism of algebras.
\end{corollary}

Now, let us give the integration homomorphisms on the derived Hall subalgebras $\mathcal {D}\mathcal {H}_q^{c_1}(\A)$ and $\mathcal {D}\mathcal {H}_q^{c_2}(\A)$. For each positive integer $t$,
let $\mathcal{T}_t$ be the $\ZZ[v,v^{-1}]$-algebra with the basis $\{X^{\alpha}~|~\alpha\in \mathbb{Z}^t\}$ and the
multiplication given by
\[X^{\alpha}\diamond X^{\beta}=X^{\alpha+\beta}.\]
It is well known that there is an isomorphism of abelian groups
$$f: K(D^b(\A))\longrightarrow K(\A)$$ defined by $f(\hat{X}_\bullet)=\sum\limits_{i\in\mathbb{Z}}(-1)^i\Dim X_i=:\Dim {X}_\bullet.$ Moreover, $\lr{X_\bullet,Y_\bullet}=\lr{\Dim {X}_\bullet,\Dim {Y}_\bullet}.$

We have the integration map \begin{equation}\int:\mathcal {D}\mathcal {H}_q^{ec}(\A)\longrightarrow\mathcal{T}_n,~~u_{X_\bullet}\mapsto X^{\Dim {X}_\bullet}.\end{equation}
In general, it is not a homomorphism of algebras, since the relation (\ref{ecr7}) is not preserved under the map $\int$. However, considering the restrictions to $\mathcal {D}\mathcal {H}_q^{c_i}(\A)$, we have the following

%\begin{lemma}\label{ceuler}
%For $j=1,2$ and any objects $X_\bullet, Y_\bullet\in C_{\A}^j$,
%$\Hom_{D^b(\A)}(X_\bullet,Y_\bullet[i])=0$
%if $|i|>1$.
%\end{lemma}
%\begin{proof}
%See the proof of \cite[Lemma 8.7]{FPZ}.
%\end{proof}

\begin{proposition}\label{jfys}
For each $i=1,2$, the integration map $$\int:\mathcal {D}\mathcal {H}_q^{c_i}(\A)\longrightarrow\mathcal{T}_n,~~u_{X_\bullet}\mapsto X^{\Dim {X}_\bullet}$$
is a homomorphism of algebras.
\end{proposition}
\begin{proof}
See the proof for the case when $i=1$ in \cite[Proposition 8.8]{FPZ}, dually, we can prove the statement for $i=2$.
\end{proof}
%\begin{remark}\label{nonint}
%The map $$\int:\mathcal {D}\mathcal {H}_q^{ec}(\A)\longrightarrow\mathcal{T}_n,~~u_{X_\bullet}\mapsto X^{\Dim {X}_\bullet}$$
%is not a homomorphism of algebras. In fact, for any $X_\bullet=I[-1]\oplus M\oplus P[1],Y_\bullet=I'[-1]\oplus N\oplus P'[1]\in C_{\A}^{e}$, the sum $\sum\limits_{i>1}(-1)^i\dim_k\Hom_{D^b(\A)}({X_\bullet},{Y_\bullet}[i])$ in the proof of \cite[Proposition 8.8]{FPZ} equals to $\dim_k\Hom_{\A}(P,I')=\lr{P,I'}.$
%More explicitly, the relation (\ref{ecr7}) is not preserved under this map $\int$.
%\end{remark}

\section{From Hall algebras to the quantum torus}\label{q-char}
In this section, we apply the comultiplications and integration maps on the subalgebras of the derived Hall algebra to get the formulas of (generalised) quantum cluster characters.

\subsection{Notations in quantum cluster algebras}
Let $Q$ be an acyclic valued quiver (cf. \cite{Rupel1,Rupel2}) with the vertex set $\{1,2,\cdots,n\}$. For each vertex $i$, let $d_i\in\mathbb{N^+}$ be the corresponding valuation. Note that each finite dimensional hereditary $k$-algebra can be obtained by taking the tensor algebra of the $k$-species $\mathfrak{S}$ associated to $Q$. We identity a $k$-species $\mathfrak{S}$ with its corresponding tensor algebra. Let $m\geq n$, we define a new quiver $\widetilde{Q}$ by attaching additional vertices $n+1,\ldots,m$ to $Q$ with the valuations $d_{n+1},\cdots, d_m$, respectively.

For each $1\leq i\leq m$, denote by $S_i$ the $i$-th simple module for $\widetilde{\mathfrak{S}}$ which is the $k$-species associated to $\widetilde{Q}$, and set $\mathcal {D}_i=\End_{\widetilde{\mathfrak{S}}}(S_i)$. Let $R(\widetilde{Q})$ and $R'(\widetilde{Q})$ be the $m\times m$ matrices with the $i$-th row and $j$-th column elements given respectively by
$$r_{ij}=\dim_{\mathcal {D}_i}\Ext_{\widetilde{\mathfrak{S}}}^1(S_j,S_i)$$
and
$$r'_{ij}=\dim_{{\mathcal {D}_i}^{op}}\Ext_{\widetilde{\mathfrak{S}}}^1(S_i,S_j),$$
where $1\leq i,j\leq m$. Define $B(\widetilde{Q})=R'(\widetilde{Q})-R(\widetilde{Q})$,
$E(\widetilde{Q})=I_m-R'(\widetilde{Q})$ and $E'(\widetilde{Q})=I_m-R(\widetilde{Q})$, where $I_m$ is the $m\times m$ identity matrix.

In what follows, we
denote by $\widetilde{R}$, $\widetilde{R}'$, $\widetilde{B}$, $\widetilde{E}$, $\widetilde{E}'$ and $\widetilde{I}$ the left $m\times n$ submatrices of $R(\widetilde{Q})$, $R'(\widetilde{Q})$, $B(\widetilde{Q})$, $E(\widetilde{Q})$, $E'(\widetilde{Q})$ and $I_m$, respectively. For a module $X$, we will always use the corresponding lowercase boldface letter ${\bf x}$
to denote its dimension vector.

From now on, let $\mathcal{A}$ (resp. $\widetilde{\mathcal{A}}$) be the category of finite dimensional left $\mathfrak{S}$ (resp. $\mathfrak{\widetilde{S}}$)-modules. We may identify $\mathcal{A}$ with the full subcategory of $\widetilde{\mathcal{A}}$ consisting of modules with supports on $Q$. For an $\mathfrak{S}$-module $X$, we also denote by $\bf{x}$ the dimension vector of $X$ viewed as an $\widetilde{\mathfrak{S}}$-module, since this should not cause confusion by the context. Thus, $E(\widetilde{Q})\x=\widetilde{E}\x$ and $E'(\widetilde{Q})\x=\widetilde{E}'{\x}$.

In what follows, we always assume that there exists a skew-symmetric $m\times m$ integral matrix $\Lambda$ such that \begin{equation}\Lambda (-B(\widetilde{Q}))=\operatorname{diag}\{d_1,\cdots, d_m\}.\end{equation} In this case,  \begin{align}\label{compatible}
\Lambda(-\widetilde{B})={D_n\choose0},\end{align} where $D_n=\operatorname{diag}\{d_1,\cdots, d_n\}$.

Note that the matrix representing the Euler form associated to $\mathfrak{\widetilde{S}}$ under the standard basis is $E(\widetilde{Q})D_m=D_mE(\widetilde{Q})$, where $D_m=\operatorname{diag}\{d_1,\cdots, d_m\}$. Then it is easy to see the following
\begin{lemma}{\rm(\cite{FPZ})}\label{sjishu} For any $\alpha,\beta\in\mathbb{Z}^m$, we have that
\begin{itemize}
\item[$(1)$]$\Lambda(B(\widetilde{Q})\alpha,E(\widetilde{Q})\beta)=\lr{\alpha,\beta}$;~~$(2)$~~$\Lambda(B(\widetilde{Q})\alpha,E'(\widetilde{Q})\beta)=\lr{\beta,\alpha}$;
\item[$(3)$]$\Lambda(B(\widetilde{Q})\alpha,B(\widetilde{Q})\beta)=\lr{\beta,\alpha}-\lr{\alpha,\beta}$;~~$(4)$~~$\Lambda(E'(\widetilde{Q})\alpha,E'(\widetilde{Q})\beta)=\Lambda(E(\widetilde{Q})\alpha,E(\widetilde{Q})\beta)$.
\end{itemize}\end{lemma}

For the simplicity of notation, in what follows, for each $\mathfrak{S}$-module or $\mathfrak{\widetilde{S}}$-module $X$, we write $^\ast\x=E(\widetilde{Q})\x$ and $\x^\ast=E'(\widetilde{Q})\x$.

\subsection{$\Lambda$-twisted versions}
In order to relate the Hall algebras with the quantum cluster algebras, we need to twist the multiplications of the Hall algebras using the bilinear form $\Lambda$ in the quantum cluster algebras.
So, let us twist the multiplication on $\mathcal {D}\mathcal {H}_q^{ec}(\widetilde{\A})$, and define $\mathcal {D}\mathcal {H}_\Lambda^{ec}(\widetilde{\A})$ to be the same module as $\mathcal {D}\mathcal {H}_q^{ec}(\widetilde{\A})$ but with the twisted
multiplication defined on basis elements by
{\begin{equation}\label{yizhitwisting1}
\begin{split}
   u_{I[-1]\oplus M\oplus P[1]}\star u_{J[-1]\oplus N\oplus Q[1]}:=
   v^{\Lambda(({\bf m}-{\bf i}-{\bf p})^\ast,(\bf{n}
   -{\bf j}-\bf{q})^\ast)}
   u_{I[-1]\oplus M\oplus P[1]}\ast u_{J[-1]\oplus N\oplus Q[1]},\end{split}
\end{equation}}where $M, N\in\widetilde{\A}$, $I, J\in\I_{\widetilde{\A}}$ and $P, Q\in\P_{\widetilde{\A}}$. We also twist the multiplication on the tensor algebra $(\mathcal {D}\mathcal {H}_q^{ec}(\widetilde{\A})\otimes\mathcal {D}\mathcal {H}_q^{ec}(\widetilde{\A}),\ast)$ by defining
\begin{equation*}\begin{split}(u_{I[-1]\oplus M\oplus P[1]}\otimes u_{J[-1]\oplus N\oplus Q[1]})\star(u_{I'[-1]\oplus M'\oplus P'[1]}\otimes u_{J'[-1]\oplus N'\oplus Q'[1]}):=\\
v^{\lambda}
(u_{I[-1]\oplus M\oplus P[1]}\otimes u_{J[-1]\oplus N\oplus Q[1]})\ast(u_{I'[-1]\oplus M'\oplus P'[1]}\otimes u_{J'[-1]\oplus N'\oplus Q'[1]}),
\end{split}
\end{equation*}where $\lambda=\Lambda(({\bf m}-{\bf i}-{\bf p}+{\bf n}-{\bf j}-{\bf q})^\ast,({\bf m}'-{\bf i}'-{\bf p}'+{\bf n}'-{\bf j}'-{\bf q}')^\ast)$, $M, N,M',N'\in\widetilde{\A}$, $I, J,I',J'\in\I_{\widetilde{\A}}$ and $P, Q,P',Q'\in\P_{\widetilde{\A}}$. Then, for each $i=1,2$, we have the subalgebra $\mathcal {D}\mathcal {H}_\Lambda^{c_i}(\widetilde{\A})$ of $\mathcal {D}\mathcal {H}_\Lambda^{ec}(\widetilde{\A})$ corresponding to $\mathcal {D}\mathcal {H}_q^{c_i}(\widetilde{\A})$ and the twisted multiplication $\star$ on $\mathcal {D}\mathcal {H}_q^{c_i}(\widetilde{\A})\otimes\mathcal {D}\mathcal {H}_q^{c_i}(\widetilde{\A})$.

Let us reformulate Proposition \ref{esubalgebra} as the following
\begin{proposition}\label{elsubalgebra}
The subalgebra $\mathcal {D}\mathcal {H}_\Lambda^{ec}(\widetilde{\A})$ is generated by the elements $\{u_{P[1]},u_{M},u_{I[-1]}~|~P\in\P_{\widetilde{\A}},M\in\widetilde{\A},I\in\I_{\widetilde{\A}}\}$, and the following relations
\begin{flalign}
&u_{P[1]}\star u_{Q[1]}=q^{\frac{1}{2}\Lambda({\bf p}^\ast,{\bf q}^\ast)}
u_{(P\oplus Q)[1]}
=q^{\Lambda({\bf p}^\ast,{\bf q}^\ast)}u_{Q[1]}\star u_{P[1]};\label{lgx2}\\
&u_{M}\star u_{P[1]}=q^{-\frac{1}{2}\Lambda({\bf m}^\ast,{\bf p}^\ast)}u_{M\oplus
P[1]};\label{lgx5}\\
&u_{P[1]}\star u_{M}=q^{-\frac{1}{2}\Lambda({\bf p}^\ast,{\bf m}^\ast)-\lr{{\bf p},{\bf m}}}
\sum\limits_{[F],[P']}|{}_{P'}\Hom_{\widetilde{\A}}(P,M)_F|u_{F\oplus P'[1]};\label{lgx7}\\
&u_{M}\star u_{N}=q^{\frac{1}{2}\Lambda({\bf m}^\ast,{\bf n}^\ast)+\lr{{\bf m},{\bf n}}}\sum_{[L]}\frac{|\mathrm{Ext}_{\widetilde{\A}}^{1}(M,N)_{L}|}{|\mathrm{Hom}_{\widetilde{\A}}(M,N)|}u_L;\label{lgx3}\\
&u_{I[-1]}\star u_{J[-1]}=
q^{\frac{1}{2}\Lambda({\bf i}^\ast,{\bf j}^\ast)}u_{(I\oplus J)[-1]}
=q^{\Lambda({\bf i}^\ast,{\bf j}^\ast)}u_{J[-1]}\star u_{I[-1]};\label{lgx1}\\
&u_{I[-1]}\star u_{M}=q^{-\frac{1}{2}\Lambda({\bf i}^\ast,{\bf m}^\ast)}u_{M\oplus
I[-1]};\label{lgx4}\\
&u_{M}\star u_{I[-1]}=q^{-\frac{1}{2}\Lambda({\bf m}^\ast,{\bf i}^\ast)-\lr{{\bf m},{\bf i}}}
\sum\limits_{[G],[I']}|{}_{G}\Hom_{\widetilde{\A}}(M,I)_{I'}|u_{G\oplus I'[-1]};\label{lgx6}\\
&u_{I[-1]}\star u_{P[1]}=q^{\Lambda({\bf i}^\ast,{\bf p}^\ast)-\lr{{\bf p},{\bf i}}}u_{P[1]}\star u_{I[-1]}=q^{\frac{1}{2}\Lambda({\bf i}^\ast,{\bf p}^\ast)}u_{I[-1]\oplus P[1]};\label{lgx8}\end{flalign}
for any $P,Q\in\P_{\widetilde{\A}}$, $M,N\in\widetilde{\A}$ and $I,J\in\I_{\widetilde{\A}}$.
\end{proposition}

The subalgebras $\mathcal {D}\mathcal {H}_\Lambda^{c_1}(\widetilde{\A})$ and $\mathcal {D}\mathcal {H}_\Lambda^{c_2}(\widetilde{\A})$ are subject to the relations (\ref{lgx2}-\ref{lgx3}) and (\ref{lgx3}-\ref{lgx6}), respectively.

Since the comultiplications $\Delta$ defined on $\mathcal {D}\mathcal {H}_q^{c_1}(\widetilde{\A})$ and $\mathcal {D}\mathcal {H}_q^{c_2}(\widetilde{\A})$ are both homogeneous, using Corollaries \ref{c1p} and \ref{c2p}, we obtain the following
\begin{lemma}\label{twyucheng12}
For each $i=1,2$, the map $\Delta:(\mathcal {D}\mathcal {H}_\Lambda^{c_i}(\widetilde{\A}),\star)\longrightarrow (\mathcal {D}\mathcal {H}_q^{c_i}(\widetilde{\A})\otimes\mathcal {D}\mathcal {H}_q^{c_i}(\widetilde{\A}),\star)$ is a homomorphism of algebras.
\end{lemma}
%\begin{proof}
%The proof is similar to that of \cite[Lemma 6.2]{FPZ}.
%\end{proof}

We also need to twist the multiplication on the tensor algebra of the torus $\mathcal{T}_m$ by defining
\begin{equation}(X^\alpha\otimes X^\beta)\star (X^\gamma\otimes X^\delta):=q^{\frac{1}{2}\Lambda((\alpha+\beta)^\ast, (\gamma+\delta)^\ast)+(\beta, \gamma)+\langle \alpha, \delta\rangle}X^{\alpha+\gamma}\otimes X^{\beta+\delta}
\end{equation} for any $\alpha, \beta, \gamma, \delta\in\mathbb{Z}^m$.
\begin{lemma}\label{twjf12}
For each $i=1,2$, the map $\int\otimes\int: (\mathcal {D}\mathcal {H}_q^{c_i}(\widetilde{\A})\otimes\mathcal {D}\mathcal {H}_q^{c_i}(\widetilde{\A}),\star)\longrightarrow(\mathcal{T}_m\otimes\mathcal{T}_m,\star)$ is a homomorphism of algebras.
\end{lemma}
\begin{proof}
By Proposition \ref{jfys}, we have the homomorphism of algebras
\begin{equation}\label{zmxy}\int\otimes\int: (\mathcal {D}\mathcal {H}_q^{c_i}(\widetilde{\A})\otimes\mathcal {D}\mathcal {H}_q^{c_i}(\widetilde{\A}),\ast)\longrightarrow(\mathcal{T}_m\otimes\mathcal{T}_m,\ast),\end{equation}
where the multiplication $\ast$ on $\mathcal{T}_m\otimes\mathcal{T}_m$ is defined by $$(X^\alpha\otimes X^\beta)\ast (X^\gamma\otimes X^\delta):=q^{(\beta, \gamma)+\langle \alpha, \delta\rangle}X^{\alpha+\gamma}\otimes X^{\beta+\delta}$$
for any $\alpha, \beta, \gamma, \delta\in\mathbb{Z}^m$. So, \begin{equation*}(X^\alpha\otimes X^\beta)\star (X^\gamma\otimes X^\delta):=v^{\Lambda((\alpha+\beta)^\ast, (\gamma+\delta)^\ast)}(X^\alpha\otimes X^\beta)\ast (X^\gamma\otimes X^\delta).
\end{equation*}
Thus, by using the coincident twisting on both sides of (\ref{zmxy}), we get the desired homomorphism of algebras.
\end{proof}

\subsection{Homomorphisms to the quantum torus}
Define the {\em quantum torus} $\mathcal{T}_\Lambda$ to be the $\ZZ[v,v^{-1}]$-algebra with the basis $\{X^{\alpha}~|~\alpha\in \mathbb{Z}^{m}\}$ and the
multiplication given by
\begin{equation}\label{ljhtorus}X^{\alpha}\star X^{\beta}=v^{\Lambda(\alpha,\beta)}X^{\alpha+\beta}.\end{equation}
%It is well known that $\mathcal{T}_\Lambda$ is an Ore domain, and thus is contained in its skew-field of fractions $\mathcal {F}$.

\begin{proposition}{\rm(\cite[Proposition 8.11]{FPZ})}\label{tzh}
The map $\mu: (\mathcal{T}_m\otimes\mathcal{T}_m,\star)\longrightarrow(\mathcal{T}_\Lambda,\star)$ defined by
$$\mu(X^{\alpha}\otimes X^{\beta})=v^{-(\alpha,\beta)-\lr{\alpha,\beta}}X^{-^\ast\alpha-\beta^\ast},$$
where $\alpha,\beta\in\mathbb{Z}^m$, is a homomorphism of algebras.
\end{proposition}
%\begin{proof}
%See the proof of \cite[Proposition 6.5]{FPZ} for a similar proof.
%\end{proof}

Given $M\in\widetilde{\A}$, we denote by $\mathrm{Gr}_{\bf{e}}M$ the set of all submodules $V$
of $M$ with $\Dim V= \bf{e}$.
Consider the map $\psi=\mu\circ(\int\otimes\int)\circ\Delta$ defined by the following commutative diagram
\begin{equation}\label{0fenjie}
\xymatrix{(\mathcal {D}\mathcal {H}_\Lambda^{ec}(\widetilde{\A}),\star)\ar[r]^-{\psi}\ar[d]_-{\Delta}&(\mathcal{T}_\Lambda,\star)\\(\mathcal {D}\mathcal {H}_q^{ec}(\widetilde{\A})\otimes\mathcal {D}\mathcal {H}_q^{ec}(\widetilde{\A}),\star)\ar[r]^-{\int\otimes\int}&(\mathcal{T}_m\otimes\mathcal{T}_m,\star).\ar[u]_-{\mu}}
\end{equation}
By the definitions of these three maps, we can directly compute $\psi(u_{I[-1]\oplus M\oplus P[1]})$, and obtain that
$$\psi(u_{I[-1]\oplus M\oplus P[1]})=\sum\limits_{\mathbf{e}}v^{\lr{{\bf p}-\mathbf{e},\mathbf{m}-\mathbf{e}-\mathbf{i}}}|\mathrm{Gr}_{\mathbf{e}}M|
X^{({\bf p}-\mathbf{e})^\ast-^\ast(\mathbf{m}-\mathbf{e}-\mathbf{i})},$$ for any $M\in\widetilde{\A}$, $I\in\I_{\widetilde{\A}}$ and $P\in\P_{\widetilde{\A}}$.

For each $i=1,2$, denote by $\psi_i$ the restriction of $\psi$ to $\mathcal {D}\mathcal {H}_\Lambda^{c_i}(\widetilde{\A})$, i.e., we have the following commutative diagram
\begin{equation}\label{fenjie}
\xymatrix{(\mathcal {D}\mathcal {H}_\Lambda^{c_i}(\widetilde{\A}),\star)\ar[r]^-{\psi_i}\ar[d]_-{\Delta}&(\mathcal{T}_\Lambda,\star)\\(\mathcal {D}\mathcal {H}_q^{c_i}(\widetilde{\A})\otimes\mathcal {D}\mathcal {H}_q^{c_i}(\widetilde{\A}),\star)\ar[r]^-{\int\otimes\int}&(\mathcal{T}_m\otimes\mathcal{T}_m,\star).\ar[u]_-{\mu}}
\end{equation}
Using Lemmas \ref{twyucheng12}, \ref{twjf12} and Proposition \ref{tzh}, we obtain that the map $\psi_i$ is a homomorphism of algebras for each $i=1,2$. To sum up, we have the following
\begin{proposition}{\rm(\cite[Theorem 8.12]{FPZ})}
The map $\psi_1:(\mathcal {D}\mathcal {H}_\Lambda^{c_1}(\widetilde{\A}),\star)\longrightarrow(\mathcal{T}_\Lambda,\star)$ is a homomorphism of algebras. Moreover, for any $M\in\widetilde{\A}$ and $P\in\P_{\widetilde{\A}}$,
$$\psi_1(u_{M\oplus P[1]})=\sum\limits_{\mathbf{e}}v^{\lr{\mathbf{p}-\mathbf{e},\mathbf{m}-\mathbf{e}}}|\mathrm{Gr}_{\mathbf{e}}M|
X^{(\mathbf{p}-\mathbf{e})^\ast-^\ast(\mathbf{m}-\mathbf{e})}.$$
\end{proposition}
\begin{proposition}
The map $\psi_2:(\mathcal {D}\mathcal {H}_\Lambda^{c_2}(\widetilde{\A}),\star)\longrightarrow(\mathcal{T}_\Lambda,\star)$ is a homomorphism of algebras. Moreover, for any $M\in\widetilde{\A}$ and $I\in\I_{\widetilde{\A}}$,
$$\psi_2(u_{M\oplus I[-1]})=\sum\limits_{\mathbf{e}}v^{-\lr{\mathbf{e},\mathbf{m}-\mathbf{e}-\mathbf{i}}}|\mathrm{Gr}_{\mathbf{e}}M|
X^{-\mathbf{e}^\ast-^\ast(\mathbf{m}-\mathbf{e}-\mathbf{i})}.$$
\end{proposition}
%\begin{proof}It is straightforward to compute $\mu\circ(\int\otimes\int)\circ\Delta(u_{M\oplus I[-1]})$ by the definitions of these three maps.\end{proof}
In fact, the map $\psi$ is also a homomorphism of algebras, i.e., we have the following
\begin{theorem}\label{mainresult}
The map $\psi:(\mathcal {D}\mathcal {H}_\Lambda^{ec}(\widetilde{\A}),\star)\longrightarrow(\mathcal{T}_\Lambda,\star)$ defined by
$$\psi(u_{I[-1]\oplus M\oplus P[1]})=\sum\limits_{\mathbf{e}}v^{\lr{{\bf p}-\mathbf{e},\mathbf{m}-\mathbf{e}-\mathbf{i}}}|\mathrm{Gr}_{\mathbf{e}}M|
X^{({\bf p}-\mathbf{e})^\ast-^\ast(\mathbf{m}-\mathbf{e}-\mathbf{i})},$$ for any $M\in\widetilde{\A}$, $I\in\I_{\widetilde{\A}}$ and $P\in\P_{\widetilde{\A}}$, is a homomorphism of algebras.
\end{theorem}
\begin{proof}
Since we have that $\psi_i$ is a homomorphism of algebras for each $i=1,2$,
it suffices to prove that $\psi$ preserves the relation (\ref{lgx8}).

In the quantum torus $\mathcal{T}_\Lambda$, noting that
\begin{flalign*}
\Lambda({\bf p}^\ast,^\ast{\bf i})=\Lambda({E}'(\widetilde{Q}){\bf p},{E}'(\widetilde{Q}){\bf i})-\Lambda({E}'(\widetilde{Q}){\bf p},B(\widetilde{Q}){\bf i})=\Lambda({\bf p}^\ast,{\bf i}^\ast)+\lr{{\bf p},{\bf i}},
\end{flalign*}
we have that
\begin{flalign*}
X^{^\ast{\bf i}}X^{{\bf p}^\ast}=q^{\Lambda(^\ast{\bf i},{\bf p}^\ast)}X^{{\bf p}^\ast}X^{^\ast{\bf i}}=q^{\Lambda({\bf i}^\ast,{\bf p}^\ast)-\lr{{\bf p},{\bf i}}}X^{{\bf p^\ast}}X^{^\ast{\bf i}}.
\end{flalign*}
Since
\begin{flalign*}X^{{\bf p}^\ast}X^{^\ast{\bf i}}=v^{\Lambda({\bf p}^\ast,^\ast{\bf i})}X^{{\bf p}^\ast+^\ast{\bf i}}=v^{\Lambda({\bf p}^\ast,{\bf i}^\ast)+\lr{{\bf p},{\bf i}}}X^{{\bf p}^\ast+^\ast{\bf i}},\end{flalign*}
we obtain that
\begin{flalign*}\psi(v^{\Lambda({\bf i}^\ast,{\bf p}^\ast)}u_{I[-1]\oplus P[1]})=
v^{\Lambda({\bf i}^\ast,{\bf p}^\ast)-\lr{{\bf p},{\bf i}}}X^{{\bf p}^\ast+^\ast{\bf i}}=
q^{\Lambda({\bf i}^\ast,{\bf p}^\ast)-\lr{{\bf p},{\bf i}}}X^{{\bf p}^\ast}X^{^\ast{\bf i}}.\end{flalign*}
Therefore, the relation (\ref{lgx8}) is preserved under $\psi$, and thus we complete the proof.
\end{proof}
%\begin{remark}
%For the homomorphism $\psi$, it is straightforward to compute and then see that we also have the factorization as that in (\ref{fenjie}), i.e., $\psi=\mu\circ(\int\otimes\int)\circ\Delta$, although as mentioned in Remark \ref{nonint}, the map $$\int:\mathcal {D}\mathcal {H}_q^{ec}(\widetilde{\A})\longrightarrow\mathcal{T}_m,~~u_{X_\bullet}\mapsto X^{\Dim {X}_\bullet}$$
%is not a homomorphism of algebras.
%\end{remark}

For any $M\in\widetilde{\A}$, $I\in\I_{\widetilde{\A}}$ and $P\in\P_{\widetilde{\A}}$, define the {\em quantum cluster character} $X_{I[-1]\oplus M\oplus P[1]}$ in the quantum torus $\mathcal{T}_\Lambda$ by
\begin{equation}\label{qcltz}X_{I[-1]\oplus M\oplus P[1]}=\sum\limits_{\mathbf{e}}v^{\lr{{\bf p}-\mathbf{e},\mathbf{m}-\mathbf{e}-\mathbf{i}}}|\mathrm{Gr}_{\mathbf{e}}M|
X^{({\bf p}-\mathbf{e})^\ast-^\ast(\mathbf{m}-\mathbf{e}-\mathbf{i})}.\end{equation}
Here the quantum cluster characters $X_M$, $X_{M\oplus I[-1]}$ and $X_{M\oplus P[1]}$ recover those given in \cite{Rupel1}, \cite{D} and \cite{DXZ}, respectively.

\begin{remark}\label{yizhi}
Let $P\in\P_{\widetilde{\A}}$ and $I=\nu(P)$, where $\nu$ is the Nakayama functor in $\widetilde{\A}$. Then $\soc I\cong \top P$. Noting that $\Dim\soc I={^\ast{\bf i}}$ and $\Dim\top P={\bf p}^\ast$, we have that ${^\ast{\bf i}}={{\bf p}^\ast}$. Thus, $X_{I[-1]}=X_{P[1]}$. However, in general, $X_{M\oplus I[-1]}\neq X_{M\oplus P[1]}$ for $M\in\widetilde{\A}$. On the other hand, if $\Hom_{\widetilde{\A}}(M,I)\cong\Hom_{\widetilde{\A}}(P,M)=0$, i.e., $\lr{{\bf m},{\bf i}}=\lr{{\bf p},{\bf m}}=0$, then $X_{M\oplus I[-1]}= X_{M\oplus P[1]}$. In fact, in this case, for any submodule $E$ of $M$, it is easy to see that $\Hom_{\widetilde{\A}}(E,I)=0$ and $\Hom_{\widetilde{\A}}(P,M/E)=0$.
\end{remark}

Using Theorem \ref{mainresult} together with Proposition \ref{elsubalgebra}, we obtain the following
\begin{corollary}\label{Hallcfgs}
We have the following cluster multiplication formulas in the quantum torus $\mathcal{T}_\Lambda$:
\begin{flalign}
&X_{P[1]} X_{Q[1]}=q^{\frac{1}{2}\Lambda({\bf p}^\ast,{\bf q}^\ast)}
X_{(P\oplus Q)[1]}
=q^{\Lambda({\bf p}^\ast,{\bf q}^\ast)}X_{Q[1]} X_{P[1]};\label{qre2}\\
&X_{M} X_{P[1]}=q^{-\frac{1}{2}\Lambda({\bf m}^\ast,{\bf p}^\ast)}X_{M\oplus
P[1]};\label{qre5}\\
&X_{P[1]} X_{M}=q^{-\frac{1}{2}\Lambda({\bf p}^\ast,{\bf m}^\ast)-\lr{{\bf p},{\bf m}}}
\sum\limits_{[F],[P']}|{}_{P'}\Hom_{\widetilde{\A}}(P,M)_F|X_{F\oplus P'[1]};\label{qre7}\\
&X_{M} X_{N}=q^{\frac{1}{2}\Lambda({\bf m}^\ast,{\bf n}^\ast)+\lr{{\bf m},{\bf n}}}\sum_{[L]}\frac{|\mathrm{Ext}_{\widetilde{\A}}^{1}(M,N)_{L}|}{|\mathrm{Hom}_{\widetilde{\A}}(M,N)|}X_L;\label{qre3}\\
&X_{I[-1]}X_{J[-1]}=
q^{\frac{1}{2}\Lambda({\bf i}^\ast,{\bf j}^\ast)}X_{(I\oplus J)[-1]}
=q^{\Lambda({\bf i}^\ast,{\bf j}^\ast)}X_{J[-1]}X_{I[-1]};\label{qre1}\\
&X_{I[-1]} X_{M}=q^{-\frac{1}{2}\Lambda({\bf i}^\ast,{\bf m}^\ast)}X_{M\oplus
I[-1]};\label{qre4}\\
&X_{M} X_{I[-1]}=q^{-\frac{1}{2}\Lambda({\bf m}^\ast,{\bf i}^\ast)-\lr{{\bf m},{\bf i}}}
\sum\limits_{[G],[I']}|{}_{G}\Hom_{\widetilde{\A}}(M,I)_{I'}|X_{G\oplus I'[-1]};\label{qre6}\\
&X_{I[-1]} X_{P[1]}=q^{\Lambda({\bf i}^\ast,{\bf p}^\ast)-\lr{{\bf p},{\bf i}}}X_{P[1]} X_{I[-1]}=q^{\frac{1}{2}\Lambda({\bf i}^\ast,{\bf p}^\ast)}X_{I[-1]\oplus P[1]};\label{qre8}\end{flalign}
for any $P,Q\in\P_{\widetilde{\A}}$, $M,N\in\widetilde{\A}$ and $I,J\in\I_{\widetilde{\A}}$.
\end{corollary}

\begin{corollary}
Let $M\in\widetilde{\A}$, $I\in\I_{\widetilde{\A}}$ and $P\in\P_{\widetilde{\A}}$.
\begin{itemize}
\item[$(1)$] If $\Hom_{\widetilde{\A}}(M,I)=0$, $X_MX^{{^\ast{\bf i}}}=q^{-\Lambda({^\ast{\bf m}},{^\ast{\bf i}})}X^{{^\ast{\bf i}}}X_M$;
\item[$(2)$] If $\Hom_{\widetilde{\A}}(P,M)=0$, $X_MX^{{{\bf p}^\ast}}=q^{-\Lambda({{\bf m}^\ast},{{\bf p}^\ast})}X^{{{\bf p}^\ast}}X_M$.
\end{itemize}
\end{corollary}
\begin{proof}
$(1)$~~Since $\Hom_{\widetilde{\A}}(M,I)=0$, we have that \begin{flalign*}X_MX^{{^\ast{\bf i}}}&=X_MX_{I[-1]}=v^{-\Lambda({{\bf m}^\ast},{{\bf i}^\ast})}X_{M\oplus I[-1]}
\\&=q^{-\Lambda({{\bf m}^\ast},{{\bf i}^\ast})}X_{I[-1]}X_M=q^{-\Lambda({^\ast{\bf m}},{^\ast{\bf i}})}X_{I[-1]}X_M.\end{flalign*}

$(2)$~~Similarly.
\end{proof}

\begin{corollary}\label{shuyu}
For any $M\in\widetilde{\A}$, $I\in\I_{\widetilde{\A}}$ and $P\in\P_{\widetilde{\A}}$, we have that
\begin{flalign*}v^{\Lambda({\bf i}^\ast,{\bf m}^\ast)}X_{I[-1]}X_{M\oplus P[1]}&=v^{\Lambda({\bf m}^\ast,{\bf p}^\ast)}X_{M\oplus I[-1]}X_{P[1]}\\
&=v^{\Lambda({\bf m}^\ast,({\bf p}-{\bf i})^\ast)}X_{I[-1]}X_{M}X_{P[1]}\\
&=v^{\Lambda({{\bf i}^\ast},{{\bf p}^\ast})}X_{I[-1]\oplus M\oplus P[1]}.\end{flalign*}
\end{corollary}
\begin{proof}
In $\mathcal {D}\mathcal {H}_q(\widetilde{\A})$, by definition, \begin{flalign*}u_{I[-1]}\ast u_{M\oplus P[1]}&=q^{\lr{-{\bf i},{\bf m}-{\bf p}}}\frac{1}{\prod\limits_{i\geq0}|\Hom_{D^b(\widetilde{\A})}(I[i-1],{M\oplus P[1]})|^{(-1)^i}}u_{I[-1]\oplus M\oplus P[1]}\\
&=u_{I[-1]\oplus M\oplus P[1]}.\end{flalign*}
Thus, in $\mathcal {D}\mathcal {H}_\Lambda^{ec}(\widetilde{\A})$, we have that
\begin{flalign*}u_{I[-1]}\star u_{M\oplus P[1]}&=v^{-\Lambda({\bf i}^\ast,({\bf m}-{\bf p})^\ast)}u_{I[-1]}\ast u_{M\oplus P[1]}\\&=v^{-\Lambda({\bf i}^\ast,({\bf m}-{\bf p})^\ast)}u_{I[-1]\oplus M\oplus P[1]}.\end{flalign*}
Hence, we have that \begin{flalign*}u_{I[-1]\oplus M\oplus P[1]}&=v^{\Lambda({\bf i}^\ast,({\bf m}-{\bf p})^\ast)}u_{I[-1]}\star u_{M\oplus P[1]}\\
&=v^{\Lambda({\bf i}^\ast,({\bf m}-{\bf p})^\ast)+\Lambda({\bf m}^\ast,{\bf p}^\ast)}u_{I[-1]}\star u_{M}\star u_{P[1]}\\
&=v^{\Lambda({\bf i}^\ast,({\bf m}-{\bf p})^\ast)+\Lambda({\bf m}^\ast,{\bf p}^\ast)-\Lambda({\bf i}^\ast,{\bf m}^\ast)}u_{M\oplus I[-1]}\star u_{P[1]}\\
&=v^{\Lambda({\bf m}^\ast,{\bf p}^\ast)-\Lambda({\bf i}^\ast,{\bf p}^\ast)}u_{M\oplus I[-1]}\star u_{P[1]},\end{flalign*}
which can imply  the desired equations under the homomorphism $\psi$.
\end{proof}

\begin{remark}
We should be reminded that all cluster multiplication formulas in Corollary \ref{Hallcfgs}, which were proved directly in the quantum cluster algebra (cf. \cite{Qin,DX,Fei,D,DSC,DXZ}), are deduced by the relations in Hall algebras.
%These formulas show that we can realize the Hall multiplications in the quantum cluster algebras using the quantum cluster characters.
\end{remark}
%In what follows, we will prove that the mutation multiplication of the quantum cluster algebras can also be deduced by the Hall multiplication.

\section{Cluster multiplication theorem, $\textrm{I}$}\label{q-cluster}
In this section, we introduce a certain quotient of the derived Hall subalgebra $\mathcal {D}\mathcal {H}_\Lambda^{ec}(\widetilde{\A})$, and then use this quotient algebra to obtain the cluster multiplication theorem for the quantum cluster characters $X_M$ and $X_{P[1]}$.

\subsection{Hall algebra $\mathcal {D}\mathcal {H}_\Lambda^{{cl}_1}(\widetilde{\A})$}
Let $\mathfrak{I}_1$ be the two-sided ideal of the derived Hall subalgebra $\mathcal {D}\mathcal {H}_\Lambda^{ec}(\widetilde{\A})$ generated by the elements in the set $$\mathcal {S}_1:=\{u_{\nu^{-1}(I)[1]}-u_{I[-1]}~|~I\in\I_{\widetilde{\A}}\}.$$ Clearly, $\mathfrak{I}_1\subseteq\Ker \psi$. Let us define \begin{equation}\mathcal {D}\mathcal {H}_\Lambda^{{cl}_1}(\widetilde{\A}):=\mathcal {D}\mathcal {H}_\Lambda^{ec}(\widetilde{\A})/\mathfrak{I}_1.\end{equation} In the following, for any $u_{I[-1]\oplus M\oplus P[1]}\in\mathcal {D}\mathcal {H}_\Lambda^{ec}(\widetilde{\A})$, we use the same notation to denote the image of $u_{I[-1]\oplus M\oplus P[1]}$ in $\mathcal {D}\mathcal {H}_\Lambda^{{cl}_1}(\widetilde{\A})$. Thus we have that $u_{P[1]}=u_{I[-1]}$ in $\mathcal {D}\mathcal {H}_\Lambda^{{cl}_1}(\widetilde{\A})$ for any $I\in\I_{\widetilde{\A}}$ and $P=\nu^{-1}(I)$. Moreover, all the relations in Proposition \ref{elsubalgebra} also hold in $\mathcal {D}\mathcal {H}_\Lambda^{{cl}_1}(\widetilde{\A})$.

\begin{lemma}\label{yinl1}
Let $M\in\widetilde{\A}, I\in\I_{\widetilde{\A}}$ and $P=\nu^{-1}(I)$. Then we have the following equations in $\mathcal {D}\mathcal {H}_\Lambda^{{cl}_1}(\widetilde{\A}):$
\begin{flalign}
u_{M\oplus I[-1]}=v^{-\lr{{\bf p},{\bf m}}}\sum\limits_{[F],[P']}|{}_{P'}\Hom_{\widetilde{\A}}(P,M)_F|u_{F\oplus P'[1]}
\end{flalign}
and \begin{flalign}
u_{M\oplus P[1]}=v^{-\lr{{\bf m},{\bf i}}}\sum\limits_{[G],[I']}|{}_{G}\Hom_{\widetilde{\A}}(M,I)_{I'}|u_{G\oplus I'[-1]}.
\end{flalign}
\end{lemma}
\begin{proof}
By the relation (\ref{lgx7}) in Proposition \ref{elsubalgebra}, we have that in $\mathcal {D}\mathcal {H}_\Lambda^{ec}(\widetilde{\A})$ $$u_{P[1]}\star u_{M}=q^{-\frac{1}{2}\Lambda({\bf p}^\ast,{\bf m}^\ast)-\lr{{\bf p},{\bf m}}}
\sum\limits_{[F],[P']}|{}_{P'}\Hom_{\widetilde{\A}}(P,M)_F|u_{F\oplus P'[1]}.$$
Noting that $u_{P[1]}=u_{I[-1]}$ in $\mathcal {D}\mathcal {H}_\Lambda^{{cl}_1}(\widetilde{\A})$, we obtain that in $\mathcal {D}\mathcal {H}_\Lambda^{{cl}_1}(\widetilde{\A})$
\begin{equation}\label{s1}u_{I[-1]}\star u_{M}=q^{-\frac{1}{2}\Lambda({\bf p}^\ast,{\bf m}^\ast)-\lr{{\bf p},{\bf m}}}
\sum\limits_{[F],[P']}|{}_{P'}\Hom_{\widetilde{\A}}(P,M)_F|u_{F\oplus P'[1]}.\end{equation}

On the other hand, by the relation (\ref{lgx4}) in Proposition \ref{elsubalgebra}, we have that $u_{I[-1]}\star u_{M}=v^{-\Lambda({\bf i}^\ast,{\bf m}^\ast)}u_{M\oplus
I[-1]}$ in $\mathcal {D}\mathcal {H}_\Lambda^{ec}(\widetilde{\A})$, which implies that in $\mathcal {D}\mathcal {H}_\Lambda^{{cl}_1}(\widetilde{\A})$
\begin{equation}\label{s2}u_{I[-1]}\star u_{M}=v^{-\Lambda({\bf i}^\ast,{\bf m}^\ast)}u_{M\oplus
I[-1]}.\end{equation}
Since \begin{flalign*}\Lambda({\bf i}^\ast,{\bf m}^\ast)&=\Lambda({E}'(\widetilde{Q}){\bf i},{E}'(\widetilde{Q}){\bf m})\\&=\Lambda({E}(\widetilde{Q}){\bf i},{E}'(\widetilde{Q}){\bf m})+\Lambda({B}(\widetilde{Q}){\bf i},{E}'(\widetilde{Q}){\bf m})\\&=\Lambda(^\ast{\bf i},{\bf m}^\ast)+\lr{{\bf m},{\bf i}}\\&=\Lambda({\bf p}^\ast,{\bf m}^\ast)+\lr{{\bf p},{\bf m}},\end{flalign*}
combining (\ref{s1}) with (\ref{s2}), we complete the proof of the first equation. Similarly, we can prove the second equation.
\end{proof}
%\begin{remark}
%Let $I\in\I_{\widetilde{\A}}$ and $P=\nu^{-1}(I)$. If $\Hom_{\widetilde{\A}}(M,I)\cong\Hom_{\widetilde{\A}}(P,M)=0$, by Lemma \ref{yinl1}, we have that $u_{M\oplus I[-1]}= u_{M\oplus P[1]}$ in $\mathcal {D}\mathcal {H}_\Lambda^{{cl}_1}(\widetilde{\A})$. This is similar to the equation $X_{M\oplus I[-1]}= X_{M\oplus P[1]}$ in Remark \ref{yizhi}.
%\end{remark}
\begin{lemma}\label{yinlx}
Let $M\in\widetilde{\A}, I\in\I_{\widetilde{\A}}$ and $P=\nu^{-1}(I)$. Then we have the following equations in $\mathcal {D}\mathcal {H}_\Lambda^{{cl}_1}(\widetilde{\A}):$
\begin{equation}\label{fc3}
\begin{split}
&(q^{\lr{{\bf m},{\bf i}}}-1)u_{M\oplus I[-1]}\\&=q^{\frac{1}{2}\lr{{\bf p},{\bf m}}}\sum\limits_{\begin{smallmatrix}[F],[P']\\P'\ncong P\end{smallmatrix}}|{}_{P'}\Hom_{\widetilde{\A}}(P,M)_F|u_{F\oplus P'[1]}+\sum\limits_{\begin{smallmatrix}[G],[I']\\I'\ncong I\end{smallmatrix}}|{}_{G}\Hom_{\widetilde{\A}}(M,I)_{I'}|u_{G\oplus I'[-1]}
\end{split}\end{equation}
and
\begin{equation}\label{fc4}
\begin{split}
&(q^{\lr{{\bf p},{\bf m}}}-1)u_{M\oplus P[1]}\\&=\sum\limits_{\begin{smallmatrix}[F],[P']\\P'\ncong P\end{smallmatrix}}|{}_{P'}\Hom_{\widetilde{\A}}(P,M)_F|u_{F\oplus P'[1]}+q^{\frac{1}{2}\lr{{\bf m},{\bf i}}}\sum\limits_{\begin{smallmatrix}[G],[I']\\I'\ncong I\end{smallmatrix}}|{}_{G}\Hom_{\widetilde{\A}}(M,I)_{I'}|u_{G\oplus I'[-1]}.
\end{split}\end{equation}
\end{lemma}
\begin{proof}
Using Lemma \ref{yinl1}, we obtain that in $\mathcal {D}\mathcal {H}_\Lambda^{{cl}_1}(\widetilde{\A}):$
\begin{flalign}\label{fc1}v^{\lr{{\bf p},{\bf m}}}u_{M\oplus I[-1]}-u_{M\oplus P[1]}=\sum\limits_{\begin{smallmatrix}[F],[P']\\P'\ncong P\end{smallmatrix}}|{}_{P'}\Hom_{\widetilde{\A}}(P,M)_F|u_{F\oplus P'[1]}\end{flalign} and
\begin{flalign}\label{fc2}v^{\lr{{\bf m},{\bf i}}}u_{M\oplus P[1]}-u_{M\oplus I[-1]}=\sum\limits_{\begin{smallmatrix}[G],[I']\\I'\ncong I\end{smallmatrix}}|{}_{G}\Hom_{\widetilde{\A}}(M,I)_{I'}|u_{G\oplus I'[-1]}.\end{flalign}
Noting that $\lr{{\bf p},{\bf m}}=\lr{{\bf m},{\bf i}}$, by the equations (\ref{fc1}) and (\ref{fc2}), we complete the proof.
\end{proof}

\begin{proposition}\label{tgmcf}
Let $M\in\widetilde{\A}, I\in\I_{\widetilde{\A}}$ and $P=\nu^{-1}(I)$. Then we have the following equations in $\mathcal {D}\mathcal {H}_\Lambda^{{cl}_1}(\widetilde{\A}):$
\begin{flalign*}
&(q^{\lr{{\bf p},{\bf m}}}-1)u_{P[1]}\star u_{M}=\\&q^{\frac{1}{2}\Lambda({\bf m}^\ast,{\bf p}^\ast)}(\sum\limits_{\begin{smallmatrix}[F],[P']\\P'\ncong P\end{smallmatrix}}|{}_{P'}\Hom_{\widetilde{\A}}(P,M)_F|u_{F\oplus P'[1]}+q^{-\frac{1}{2}\lr{{\bf m},{\bf i}}}\sum\limits_{\begin{smallmatrix}[G],[I']\\I'\ncong I\end{smallmatrix}}|{}_{G}\Hom_{\widetilde{\A}}(M,I)_{I'}|u_{G\oplus I'[-1]})
\end{flalign*}
and \begin{flalign*}
&(q^{\lr{{\bf m},{\bf i}}}-1)u_{M}\star u_{I[-1]}=\\&q^{\frac{1}{2}\Lambda({\bf i}^\ast,{\bf m}^\ast)}(q^{-\frac{1}{2}\lr{{\bf p},{\bf m}}}\sum\limits_{\begin{smallmatrix}[F],[P']\\P'\ncong P\end{smallmatrix}}|{}_{P'}\Hom_{\widetilde{\A}}(P,M)_F|u_{F\oplus P'[1]}+\sum\limits_{\begin{smallmatrix}[G],[I']\\I'\ncong I\end{smallmatrix}}|{}_{G}\Hom_{\widetilde{\A}}(M,I)_{I'}|u_{G\oplus I'[-1]}).
\end{flalign*}
\end{proposition}
\begin{proof}
Noting that $u_{P[1]}=u_{I[-1]}$ in $\mathcal {D}\mathcal {H}_\Lambda^{{cl}_1}(\widetilde{\A})$, we obtain that in $\mathcal {D}\mathcal {H}_\Lambda^{{cl}_1}(\widetilde{\A})$
\begin{flalign*}
(q^{\lr{{\bf p},{\bf m}}}-1)u_{P[1]}\star u_{M}&=(q^{\lr{{\bf m},{\bf i}}}-1)u_{I[-1]}\star u_{M}\\
&=(q^{\lr{{\bf m},{\bf i}}}-1)v^{-\Lambda({\bf i}^\ast,{\bf m}^\ast)}u_{M\oplus I[-1]}\\
&=(q^{\lr{{\bf m},{\bf i}}}-1)v^{\Lambda({\bf m}^\ast,{\bf p}^\ast)-\lr{{\bf m},{\bf i}}}u_{M\oplus I[-1]}.\end{flalign*}
Then using the equation (\ref{fc3}), we complete the proof of the first equation.  Noting that $\Lambda({\bf m}^\ast,{\bf p}^\ast)=\Lambda(^\ast{\bf m},{^\ast{\bf i}})+\lr{{\bf m},{\bf i}}=\Lambda({\bf m}^\ast,{{\bf i}^\ast})+\lr{{\bf p},{\bf m}}$, similarly, we
can prove the second equation.
\end{proof}

\subsection{Cluster multiplication formulas via $\mathcal {D}\mathcal {H}_\Lambda^{{cl}_1}(\widetilde{\A})$}
Since $\mathfrak{I}_1\subseteq\Ker \psi$, the homomorphism of algebras $\psi:\mathcal {D}\mathcal {H}_\Lambda^{ec}(\widetilde{\A})\longrightarrow\mathcal{T}_\Lambda$ induces a homomorphism of algebras \begin{equation}\label{xuyaomap}\varphi_1:\mathcal {D}\mathcal {H}_\Lambda^{{cl}_1}(\widetilde{\A})\longrightarrow\mathcal{T}_\Lambda.\end{equation}

Let $\mathcal{A}\mathcal{H}_q^{\circ}(\widetilde{Q})$ be the subalgebra of $\mathcal{T}_\Lambda$ generated by the quantum cluster characters $\{X_M, X_{P[1]}~|~M\in\widetilde{\A}, P\in\P_{\widetilde{\A}}\}$, which also equals to that generated by $\{X_M, X_{I[-1]}~|~M\in\widetilde{\A}, I\in\I_{\widetilde{\A}}\}$.
\begin{proposition}\label{morphism1}
The map $\varphi_1:\mathcal {D}\mathcal {H}_\Lambda^{{cl}_1}(\widetilde{\A})\longrightarrow\mathcal{A}\mathcal{H}_q^{\circ}(\widetilde{Q})$ defined by
$$\varphi_1(u_{I[-1]\oplus M\oplus P[1]})=X_{I[-1]\oplus M\oplus P[1]},$$ for any $M\in\widetilde{\A}$, $I\in\I_{\widetilde{\A}}$ and $P\in\P_{\widetilde{\A}}$, is a surjective homomorphism of algebras.
\end{proposition}
\begin{proof}
We only need to note that $X_{I[-1]\oplus M\oplus P[1]}\in\mathcal{A}\mathcal{H}_q^{\circ}(\widetilde{Q})$ by Corollary \ref{shuyu}.
\end{proof}

Let $\mathcal{A}\mathcal{H}_q({Q})$ be the subalgebra of $\mathcal{A}\mathcal{H}_q^{\circ}(\widetilde{Q})$ generated by the quantum cluster characters $\{X_M, X_{P[1]}~|~M\in{\A}, P\in\P_{\widetilde{\A}}\}$, which also equals to that generated by $\{X_M, X_{I[-1]}~|~M\in{\A}, I\in\I_{\widetilde{\A}}\}$. Let $\mathcal {D}\mathcal {H}_\Lambda^{{\tilde{c}l}_1}({\A})$ be the subalgebra of $\mathcal {D}\mathcal {H}_\Lambda^{{cl}_1}(\widetilde{\A})$ spanned by the elements $\{u_{I[-1]\oplus M\oplus P[1]}~|~~M\in{\A},I\in\I_{\widetilde{\A}},P\in\P_{\widetilde{\A}}\}$.

\begin{proposition}
The map $\tilde{\varphi}_1:\mathcal {D}\mathcal {H}_\Lambda^{{\tilde{c}l}_1}({\A})\longrightarrow\mathcal{A}\mathcal{H}_q({Q})$ defined by
$$\tilde{\varphi}_1(u_{I[-1]\oplus M\oplus P[1]})=X_{I[-1]\oplus M\oplus P[1]},$$ for any $M\in{\A}$, $I\in\I_{\widetilde{\A}}$ and $P\in\P_{\widetilde{\A}}$, is a surjective homomorphism of algebras.
\end{proposition}

\begin{remark}The quantum cluster algebras of $Q$ with coefficients are just the subalgebras of those of $\widetilde{Q}$ without coefficients.
Therefore, in what follows, we focus on studying the multiplication formulas in $\mathcal{A}\mathcal{H}_q^{\circ}(\widetilde{Q})$. To get the corresponding formulas in $\mathcal{A}\mathcal{H}_q({Q})$, we only need to restrict the modules $M$ that we consider to $\A$.\end{remark}

Each formula in Lemma \ref{yinl1}, Lemma \ref{yinlx} and Proposition \ref{tgmcf} is mapped to a corresponding formula in $\mathcal{T}_\Lambda$ under the homomorphism (\ref{xuyaomap}). In order to save space, we only write the formula corresponding to Proposition \ref{tgmcf} as follows:
\begin{theorem}\label{dyggs}
Let $M\in\widetilde{\A}, I\in\I_{\widetilde{\A}}$ and $P=\nu^{-1}(I)$. Then we have the following equations in $\mathcal{T}_\Lambda:$
\begin{equation*}\label{xjyan1}\begin{split}
&(q^{\lr{{\bf p},{\bf m}}}-1)X_{P[1]} X_{M}=\\&q^{\frac{1}{2}\Lambda({\bf m}^\ast,{\bf p}^\ast)}(\sum\limits_{\begin{smallmatrix}[F],[P']\\P'\ncong P\end{smallmatrix}}|{}_{P'}\Hom_{\widetilde{\A}}(P,M)_F|X_{F\oplus P'[1]}+q^{-\frac{1}{2}\lr{{\bf m},{\bf i}}}\sum\limits_{\begin{smallmatrix}[G],[I']\\I'\ncong I\end{smallmatrix}}|{}_{G}\Hom_{\widetilde{\A}}(M,I)_{I'}|X_{G\oplus I'[-1]})
\end{split}\end{equation*}
and \begin{flalign*}
&(q^{\lr{{\bf m},{\bf i}}}-1)X_{M} X_{I[-1]}=\\&q^{\frac{1}{2}\Lambda({\bf i}^\ast,{\bf m}^\ast)}(q^{-\frac{1}{2}\lr{{\bf p},{\bf m}}}\sum\limits_{\begin{smallmatrix}[F],[P']\\P'\ncong P\end{smallmatrix}}|{}_{P'}\Hom_{\widetilde{\A}}(P,M)_F|X_{F\oplus P'[1]}+\sum\limits_{\begin{smallmatrix}[G],[I']\\I'\ncong I\end{smallmatrix}}|{}_{G}\Hom_{\widetilde{\A}}(M,I)_{I'}|X_{G\oplus I'[-1]}).
\end{flalign*}
\end{theorem}
The following corollary is a generalization of \cite[Theorem 4.8]{Rupel1}.
\begin{corollary}\label{mcf1}
Let $M\in\widetilde{\A}, I\in\I_{\widetilde{\A}}$ and $P=\nu^{-1}(I)$. Assume that there exist unique (up to scalar) morphisms $f\in\Hom_{\widetilde{\A}}(M,I)$ and $g\in\Hom_{\widetilde{\A}}(P,M)$, in particular, $$\dim_{{\rm End}(I)}\Hom_{\widetilde{\A}}(M,I)=\dim_{{\rm End}(P)}\Hom_{\widetilde{\A}}(P,M)=1.$$ Then we have the following equations in $\mathcal{T}_\Lambda:$
\begin{flalign*}
X_{P[1]} X_M=q^{\frac{1}{2}\Lambda({\bf m}^\ast,{\bf p}^\ast)}(X_{F\oplus P'[1]}+q^{-\frac{1}{2}\lr{{\bf i},{\bf i}}}X_{G\oplus I'[-1]})
\end{flalign*}
and
\begin{flalign*}
X_M X_{I[-1]}=q^{\frac{1}{2}\Lambda({\bf i}^\ast,{\bf m}^\ast)}(q^{-\frac{1}{2}\lr{{\bf p},{\bf p}}}X_{F\oplus P'[1]}+X_{G\oplus I'[-1]}),
\end{flalign*}
where $G=\Ker f, I'=\Coker f$, $P'=\Ker g, F=\Coker g$.
\end{corollary}
\begin{proof}
Since $\dim_{{\rm End}(I)}\Hom_{\widetilde{\A}}(M,I)=\dim_{{\rm End}(P)}\Hom_{\widetilde{\A}}(P,M)=1$, we obtain that $\lr{{\bf p},{\bf m}}=\dim_k\Hom_{\widetilde{\A}}(P,M)=\dim_k\End(P)=\lr{{\bf p},{\bf p}}$, similarly, $\lr{{\bf m},{\bf i}}=\lr{{\bf i},{\bf i}}$. Noting that $|{}_{P'}\Hom_{\widetilde{\A}}(P,M)_F|=q^{\lr{{\bf p},{\bf p}}}-1$, $|{}_{G}\Hom_{\widetilde{\A}}(M,I)_{I'}|=q^{\lr{{\bf i},{\bf i}}}-1$ and $\End(P)\cong\End(I)$, we complete the proof.
\end{proof}

\section{Multiplication formulas in quotients of Hall algebras}
In this section, we introduce another subalgebra of the derived Hall algebra and show a multiplication formula in its certain quotient.

\subsection{Hall algebra $\mathcal {D}\mathcal {H}_\Lambda^{{cl}_2}(\widetilde{\A})$}
Let $\mathcal {D}\mathcal {H}_q^{\langle0,1]}(\widetilde{\A})$ be the subalgebra of the derived Hall algebra $\mathcal {D}\mathcal {H}_q(\widetilde{\A})$ spanned by all elements $u_{M[-1]\oplus N\oplus P[1]}$ with $M,N\in\widetilde{\A}$ and $P\in\P_{\widetilde{\A}}$. Define $\mathcal {D}\mathcal {H}_\Lambda^{\langle0,1]}(\widetilde{\A})$
to be the same module as $\mathcal {D}\mathcal {H}_q^{\langle0,1]}(\widetilde{\A})$ but with the twisted
multiplication defined on basis elements by
{\begin{equation}
\begin{split}
   u_{M[-1]\oplus N\oplus P[1]}\star u_{X[-1]\oplus Y\oplus Q[1]}:=
   v^{\Lambda(({\bf n}-{\bf m}-{\bf p})^\ast,(\bf{y}
   -{\bf x}-\bf{q})^\ast)}
   u_{M[-1]\oplus N\oplus P[1]}\ast u_{X[-1]\oplus Y\oplus Q[1]},\end{split}
\end{equation}}where $M, N, X, Y\in\widetilde{\A}$ and $P, Q\in\I_{\widetilde{\A}}$.
Then by Proposition \ref{twistderived}, we have the following
\begin{proposition}
The subalgebra $\mathcal {D}\mathcal {H}_\Lambda^{\langle0,1]}(\widetilde{\A})$ is generated by the elements in the set $\{u_{M[-1]},u_{N},u_{P[1]}~|~M,N\in\widetilde{\A},P\in\P_{\widetilde{\A}}\}$, and the following relations
\begin{flalign}
&u_{M[i]}\star u_{N[i]}=q^{\frac{1}{2}\Lambda({\bf m}^\ast,{\bf n}^\ast)+\lr{{\bf m},{\bf n}}}\sum_{[L]}\frac{|\mathrm{Ext}_{\widetilde{\A}}^{1}(M,N)_{L}|}{|\mathrm{Hom}_{\widetilde{\A}}(M,N)|}u_{L[i]},~~i=-1,0;\\\label{Hallcheng}
&u_{P[1]}\star u_{Q[1]}=q^{\frac{1}{2}\Lambda({\bf p}^\ast,{\bf q}^\ast)}
u_{(P\oplus Q)[1]}
=q^{\Lambda({\bf p}^\ast,{\bf q}^\ast)}u_{Q[1]}\star u_{P[1]};\\
&u_{M[-1]}\star u_{N}=q^{-\frac{1}{2}\Lambda({\bf m}^\ast,{\bf n}^\ast)}u_{N\oplus
M[-1]};\\
&u_{N}\star u_{P[1]}=q^{-\frac{1}{2}\Lambda({\bf n}^\ast,{\bf p}^\ast)}u_{N\oplus
P[1]};\\
&u_{N}\star u_{M[-1]}=q^{-\frac{1}{2}\Lambda({\bf n}^\ast,{\bf m}^\ast)-\lr{{\bf n},{\bf m}}}
\sum\limits_{[X],[Y]}|{}_{X}\Hom_{\widetilde{\A}}(N,M)_{Y}|u_{X\oplus Y[-1]};\\
&u_{P[1]}\star u_{N}=q^{-\frac{1}{2}\Lambda({\bf p}^\ast,{\bf n}^\ast)-\lr{{\bf p},{\bf n}}}
\sum\limits_{[Q],[F]}|{}_{Q}\Hom_{\widetilde{\A}}(P,N)_F|u_{F\oplus Q[1]};\end{flalign}\begin{flalign}
u_{M[-1]}\star u_{P[1]}=q^{\Lambda({\bf m}^\ast,{\bf p}^\ast)-\lr{{\bf p},{\bf m}}}u_{P[1]}\star u_{M[-1]}=q^{\frac{1}{2}\Lambda({\bf m}^\ast,{\bf p}^\ast)}u_{M[-1]\oplus P[1]};\end{flalign}
for any $M,N\in\widetilde{\A}$ and $P,Q\in\P_{\widetilde{\A}}$.
\end{proposition}

Let $\mathfrak{I}_2$ be the two-sided ideal of the derived Hall subalgebra $\mathcal {D}\mathcal {H}_\Lambda^{\langle0,1]}(\widetilde{\A})$ generated by the elements in the set $$\mathcal {S}_2:=\{u_{\uptau^{-1}M}-u_{M[-1]}~|~M\in\widetilde{\A}\},$$ where $\uptau$ is the Auslander-Reiten translation in $D^b(\widetilde{\A})$.
Let us define \begin{equation}\mathcal {D}\mathcal {H}_\Lambda^{{cl}_2}(\widetilde{\A}):=\mathcal {D}\mathcal {H}_\Lambda^{\langle0,1]}(\widetilde{\A})/\mathfrak{I}_2.\end{equation}
\begin{remark}\label{gxyz}
For $I\in\I_{\widetilde{\A}}$, we have that $\uptau^{-1}(I)=P[1]$ in $D^b(\widetilde{\A})$, where $P=\nu^{-1}(I)$. Thus, $u_{\uptau^{-1}I}-u_{I[-1]}=u_{P[1]}-u_{I[-1]}$, which coincides with the relation in $\mathcal{S}_1$.
\end{remark}

In what follows, we denote by $\tau$ the Auslander-Reiten translation in $\widetilde{\A}$. Then for any $M\in\widetilde{\A}$, if $M$ has no nonzero projective (resp. injective) direct summands, then we have that $\uptau M=\tau M$ (resp. $\uptau^{-1} M=\tau^{-1} M$).

\begin{theorem}\label{twoalg}
The natural embedding of algebras $\xymatrix{\lambda:\mathcal {D}\mathcal {H}_\Lambda^{ec}(\widetilde{\A})\ar@{^{(}->}[r]&\mathcal {D}\mathcal {H}_\Lambda^{\langle0,1]}(\widetilde{\A})}$ induces a surjective homomorphism of algebras $\theta:\mathcal {D}\mathcal {H}_\Lambda^{{cl}_1}(\widetilde{\A})\longrightarrow \mathcal {D}\mathcal {H}_\Lambda^{{cl}_2}(\widetilde{\A})$.
\end{theorem}
\begin{proof}
By Remark \ref{gxyz}, $\mathcal{S}_1\subseteq\mathcal{S}_2$, so $\lambda$ induces a homomorphism of algebras
$$\theta:\mathcal {D}\mathcal {H}_\Lambda^{{cl}_1}(\widetilde{\A})\longrightarrow \mathcal {D}\mathcal {H}_\Lambda^{{cl}_2}(\widetilde{\A}).$$

For any $u_{M[-1]\oplus N\oplus P[1]}\in\mathcal {D}\mathcal {H}_\Lambda^{{cl}_2}(\widetilde{\A})$, write $M=\bar{M}\oplus I$ such that $\bar{M}$ has no injective direct summands, then it is straightforward to check that $u_{\bar{M}[-1]}\ast u_{I[-1]\oplus N\oplus P[1]}=u_{M[-1]\oplus N\oplus P[1]}$ in $\mathcal {D}\mathcal {H}_q^{\langle0,1]}(\widetilde{\A})$.
Thus, we have that in $\mathcal {D}\mathcal {H}_\Lambda^{{cl}_2}(\widetilde{\A})$
\begin{flalign*}u_{\tau^{-1}\bar{M}}\star u_{I[-1]\oplus N\oplus P[1]}&=u_{\bar{M}[-1]}\star u_{I[-1]\oplus N\oplus P[1]}\\&=v^{-\Lambda({\bar{{\bf m}}}^\ast,({\bf n}-{\bf i}-{\bf p})^\ast)}u_{M[-1]\oplus N\oplus P[1]}.\end{flalign*}
Then $\theta(v^{\Lambda({\bar{{\bf m}}}^\ast,({\bf n}-{\bf i}-{\bf p})^\ast)}u_{\tau^{-1}\bar{M}}\star u_{I[-1]\oplus N\oplus P[1]})=u_{M[-1]\oplus N\oplus P[1]}$, i.e., $\theta$ is surjective.
\end{proof}

\subsection{Multiplication formulas in $\mathcal {D}\mathcal {H}_\Lambda^{{cl}_2}(\widetilde{\A})$}\label{jihao}
Given $M,N\in\widetilde{\A}$, where $M=M'\oplus P'$ and $P'$ is the maximal projective direct summand of $M$. For each morphism $\theta:N\longrightarrow\tau M'$, we have an exact sequence
\begin{equation}\label{zhxl}
\xymatrix{0\ar[r]&D\ar[r]&N\ar[r]^-\theta&\tau M'\ar[r]&\tau A'\oplus I\ar[r]&0}
\end{equation}where $D=\Ker \theta, \tau A'\oplus I=\Coker \theta$, $I\in\I_{\widetilde{\A}}$, and $A'$ has no nonzero projective direct summands. Note that by the Auslander-Reiten formula we have that $\Ext^1_{\widetilde{\A}}(M,N)\cong\Hom_{\widetilde{\A}}(N,\tau M')$.
For the simplicity of notation, we set $[M,N]^1:=\dim_k\Ext^1_{\widetilde{\A}}(M,N)$ for any $M,N\in\widetilde{\A}$.

Firstly, let us give the following lemma, which is significant in the later calculations.
\begin{lemma}
For any $M,N\in\widetilde{\A}$, we have that $\lr{{\bf m},\uptau{\bf n}}=-\lr{{\bf n},{\bf m}}$ and ${^*\uptau{\bf n}}=-{\bf n}^*$.
\end{lemma}
\begin{proof}
For the convenience of readers, we give a simple proof.
Recall that the matrix representing the Euler form under the standard basis is $E(\widetilde{Q})D_m=D_mE(\widetilde{Q})$. Then the second equation follows from the first one.
So we only prove the first one.

If $N$ has no nonzero projective direct summands, it follows from the Auslander-Reiten formulas in $\widetilde{\A}$ (see for example, \cite[Proposition 2.4]{Rupel2}).

If $N=P$ is projective, $\uptau P=I[-1]$, where $I=\nu(P)$. Then we have that $\lr{{\bf m},\uptau{\bf n}}=\lr{{\bf m},-{\bf i}}=-\lr{{\bf p},{\bf m}}.$
\end{proof}
\begin{theorem}\label{ddl}
Let $M,N\in\widetilde{\A}$ such that $\Ext_{\widetilde{\A}}^1(N,M)=0$. Then we have the following equations in $\mathcal {D}\mathcal {H}_\Lambda^{{cl}_2}(\widetilde{\A}):$
\begin{equation}\label{halltimes}
\begin{split}
&(q^{[M,N]^1}-1)u_M\star u_N=q^{\frac{1}{2}\Lambda({\bf m}^*,{\bf n}^*)}\sum\limits_{E\ncong M\oplus N}
|\Ext_{\widetilde{\A}}^1(M,N)_E|u_E+\\&\sum\limits_{\begin{smallmatrix}D,A,I\\D\ncong N\end{smallmatrix}}q^{\frac{1}{2}\Lambda(({\bf m}-{{\bf a}})^*,{({\bf n}+{\bf a})}^*)+\frac{1}{2}\lr{{{\bf m}}-{{\bf a}},{\bf n}}}|_D\Hom_{\widetilde{\A}}(N,\tau M)_{\tau A\oplus I}|u_{A}\star u_{D\oplus I[-1]}
\end{split}\end{equation}
and
\begin{equation}
\begin{split}
&(q^{[M,N]^1}-1)u_{M\oplus N}=\sum\limits_{E\ncong M\oplus N}
|\Ext_{\widetilde{\A}}^1(M,N)_E|u_E+\\&\sum\limits_{\begin{smallmatrix}D,A,I\\D\ncong N\end{smallmatrix}}q^{\frac{1}{2}\Lambda(({\bf m}+{\bf n}-{{\bf a}})^*,{{\bf a}}^*)+\frac{1}{2}\lr{{{\bf m}}-{{\bf a}},{\bf n}}+[M,N]^1}|_D\Hom_{\widetilde{\A}}(N,\tau M)_{\tau A\oplus I}|u_{A}\star u_{D\oplus I[-1]},
\end{split}
\end{equation}
where each $A$ has the same maximal projective direct summand as $M$ in both equations.
\end{theorem}
\begin{proof}
We assume that $\Ext_{\widetilde{\A}}^1(M,N)\neq0$, otherwise, the two equations are trivial.

Write $M=M'\oplus P'$ such that $P'$ is the maximal projective direct summand of $M$.
Set $\tilde{M}=\tau M'$. On the one hand,
\begin{equation}\label{nm'3}
\begin{split}
&u_{\tilde{M}[-1]}\star u_N=u_{\tau^{-1}\tilde{M}}\star u_N=u_{M'}\star u_N\\
&=q^{\frac{1}{2}\Lambda({{\bf m}'}^*,{\bf n}^*)+\lr{{\bf m}',{\bf n}}}\sum\limits_{L\ncong M'\oplus N}\frac{|\Ext_{\widetilde{\A}}^1(M',N)_L|}{|\Hom_{\widetilde{\A}}(M',N)|}u_L+q^{\frac{1}{2}\Lambda({{\bf m}'}^*,{\bf n}^*)-[M',N]^1}u_{M'\oplus N}
\end{split}\end{equation}
and \begin{equation}\label{nm'4}
u_{\tilde{M}[-1]}\star u_N=q^{-\frac{1}{2}\Lambda({\tilde{{\bf m}}}^*,{\bf n}^*)}u_{N\oplus \tilde{M}[-1]}.\end{equation}
Thus, combining (\ref{nm'3}) with (\ref{nm'4}), we obtain that
\begin{equation}\label{fangchen1}
\begin{split}
&q^{-\frac{1}{2}\Lambda({\tilde{{\bf m}}}^*,{\bf n}^*)}u_{N\oplus \tilde{M}[-1]}-q^{\frac{1}{2}\Lambda({{\bf m}'}^*,{\bf n}^*)-[M',N]^1}u_{M'\oplus N}\\&=q^{\frac{1}{2}\Lambda({{\bf m}'}^*,{\bf n}^*)-[M',N]^1}\sum\limits_{L\ncong M'\oplus N}{|\Ext_{\widetilde{\A}}^1(M',N)_L|}u_L.
\end{split}\end{equation}

On the other hand, since $\Ext_{\widetilde{\A}}^1(N,M)=0$ implies $\Ext_{\widetilde{\A}}^1(N,M')=0$, we have that
\begin{flalign}\label{nm'1}
u_N\star u_{M'}=q^{\frac{1}{2}\Lambda({\bf n}^*,{{\bf m}'}^*)}u_{M'\oplus N}.
\end{flalign}
Meanwhile, we also have that
\begin{equation}\label{nm'2}
\begin{split}
&u_N\star u_{M'}=u_N\star u_{\tau^{-1}\tilde{M}}=u_N\star u_{\tilde{M}[-1]}\\&=q^{-\frac{1}{2}\Lambda({\bf n}^*,{\tilde{{\bf m}}}^*)-\lr{{\bf n},\tilde{{\bf m}}}}(u_{N\oplus \tilde{M}[-1]}+\sum\limits_{\begin{smallmatrix}D,A',I\\D\ncong N\end{smallmatrix}}|_D\Hom_{\widetilde{\A}}(N,\tau M')_{\tau A'\oplus I}|u_{D\oplus (\tau A'\oplus I)[-1]})\\
&=q^{t_0}(u_{N\oplus \tilde{M}[-1]}+\sum\limits_{\begin{smallmatrix}D,A',I\\D\ncong N\end{smallmatrix}}|_D\Hom_{\widetilde{\A}}(N,\tau M')_{\tau A'\oplus I}|q^{\frac{1}{2}\Lambda((\tau {\bf a}'+{\bf i})^*,{\bf d}^*)}u_{(\tau A'\oplus I)[-1]}\star u_D)\\
&=q^{t_0}(u_{N\oplus \tilde{M}[-1]}+\sum\limits_{\begin{smallmatrix}D,A',I\\D\ncong N\end{smallmatrix}}|_D\Hom_{\widetilde{\A}}(N,\tau M')_{\tau A'\oplus I}|q^{\frac{1}{2}\Lambda((\tau {\bf a}'+{\bf i})^*,{\bf d}^*)}u_{A'\oplus P[1]}\star u_D)
\end{split}\end{equation}
where $t_0=-\frac{1}{2}\Lambda({\bf n}^*,{\tilde{{\bf m}}}^*)-\lr{{\bf n},\tilde{{\bf m}}}$ and $P=\nu^{-1}(I)$ for each $I$ in the last sum in (\ref{nm'2}).
Thus, combining (\ref{nm'1}) with (\ref{nm'2}), we obtain that
\begin{equation}\label{fangchen2}
\begin{split}
&q^{\frac{1}{2}\Lambda({\bf n}^*,{{\bf m}'}^*)+\frac{1}{2}\Lambda({\bf n}^*,{\tilde{{\bf m}}}^*)+\lr{{\bf n},\tilde{{\bf m}}}}u_{M'\oplus N}-u_{N\oplus \tilde{M}[-1]}\\&=\sum\limits_{\begin{smallmatrix}D,A',I\\D\ncong N\end{smallmatrix}}|_D\Hom_{\widetilde{\A}}(N,\tau M')_{\tau A'\oplus I}|q^{\frac{1}{2}\Lambda((\tau {\bf a}'+{\bf i})^*,{\bf d}^*)}u_{A'\oplus P[1]}\star u_D.
\end{split}\end{equation}

Noting that \begin{flalign*}&\Lambda({\bf n}^*,{\tilde{{\bf m}}}^*)=\Lambda(^*{\bf n},{^*{\tilde{{\bf m}}}})=\Lambda(^*{\bf n},{^*(\tau{\bf m}')})=\Lambda(^*{\bf n},-{{\bf m}'}^*)=\Lambda({{\bf m}'}^*,{^*{\bf n}})\\&=\Lambda({{\bf m}'}^*,{{\bf n}^*})-\Lambda(E'(\widetilde{Q}){\bf m}',B(\widetilde{Q}){\bf n})=\Lambda({{\bf m}'}^*,{\bf n}^*)+\lr{{\bf m}',{\bf n}},\end{flalign*}
$\lr{{\bf n},{\tilde{{\bf m}}}}=\lr{{\bf n},\tau{\bf m}'}=-\lr{{\bf m}',{\bf n}}$~and~$[M',N]^1=[M,N]^1$, we rewrite (\ref{fangchen1}) and (\ref{fangchen2}) as
\begin{equation}\label{fangchen3}
\begin{split}q^{\frac{1}{2}\lr{{\bf m}',{\bf n}}+[M,N]^1}u_{N\oplus \tilde{M}[-1]}-u_{M'\oplus N}
=\sum\limits_{L\ncong M'\oplus N}{|\Ext_{\widetilde{\A}}^1(M',N)_L|}u_L
\end{split}\end{equation}
and
\begin{equation}\label{fangchen4}
\begin{split}&q^{-\frac{1}{2}\lr{{\bf m}',{\bf n}}}u_{M'\oplus N}-u_{N\oplus \tilde{M}[-1]}\\&
=\sum\limits_{\begin{smallmatrix}D,A',I\\D\ncong N\end{smallmatrix}}|_D\Hom_{\widetilde{\A}}(N,\tau M')_{\tau A'\oplus I}|q^{\frac{1}{2}\Lambda((\tau {\bf a}'+{\bf i})^*,{\bf d}^*)}u_{A'\oplus P[1]}\star u_D,
\end{split}\end{equation}
respectively. Then combining (\ref{fangchen3}) with (\ref{fangchen4}), we work out that
\begin{equation}\label{zjjg1}
\begin{split}
(q^{[M,N]^1}-1)&u_{N\oplus \tilde{M}[-1]}=q^{-\frac{1}{2}\lr{{\bf m}',{\bf n}}}\sum\limits_{L\ncong M'\oplus N}{|\Ext_{\widetilde{\A}}^1(M',N)_L|}u_L+\\&\sum\limits_{\begin{smallmatrix}D,A',I\\D\ncong N\end{smallmatrix}}|_D\Hom_{\widetilde{\A}}(N,\tau M')_{\tau A'\oplus I}|q^{\frac{1}{2}\Lambda((\tau {\bf a}'+{\bf i})^*,{\bf d}^*)}u_{A'\oplus P[1]}\star u_D
\end{split}\end{equation}
and
\begin{equation}\label{zjjg2}
\begin{split}
&(q^{[M,N]^1}-1)u_{M'\oplus N}=\sum\limits_{L\ncong M'\oplus N}{|\Ext_{\widetilde{\A}}^1(M',N)_L|}u_L+\\&q^{\frac{1}{2}\lr{{\bf m}',{\bf n}}+[M,N]^1}\sum\limits_{\begin{smallmatrix}D,A',I\\D\ncong N\end{smallmatrix}}|_D\Hom_{\widetilde{\A}}(N,\tau M')_{\tau A'\oplus I}|q^{\frac{1}{2}\Lambda((\tau {\bf a}'+{\bf i})^*,{\bf d}^*)}u_{A'\oplus P[1]}\star u_D.
\end{split}\end{equation}

By (\ref{nm'4}) and (\ref{nm'3}), we have that
\begin{equation}\label{dairu1}
\begin{split}
u_{N\oplus \tilde{M}[-1]}&=q^{\frac{1}{2}\Lambda({\tilde{{\bf m}}}^*,{\bf n}^*)}u_{\tilde{M}[-1]}\star u_N\\
&=q^{-\frac{1}{2}\Lambda({{\bf m}'}^*,{\bf n}^*)-\frac{1}{2}\lr{{\bf m}',{\bf n}}}u_{M'}\star u_{N}.
\end{split}\end{equation}
Substituting (\ref{dairu1}) into (\ref{zjjg1}), we obtain that
\begin{equation}\label{zhds1}
\begin{split}
&(q^{[M,N]^1}-1)u_{M'}\star u_{N}=q^{\frac{1}{2}\Lambda({{\bf m}'}^*,{\bf n}^*)}(\sum\limits_{L\ncong M'\oplus N}{|\Ext_{\widetilde{\A}}^1(M',N)_L|}u_L+\\&q^{\frac{1}{2}\lr{{\bf m}',{\bf n}}}\sum\limits_{\begin{smallmatrix}D,A',I\\D\ncong N\end{smallmatrix}}|_D\Hom_{\widetilde{\A}}(N,\tau M')_{\tau A'\oplus I}|q^{\frac{1}{2}\Lambda((\tau {\bf a}'+{\bf i})^*,{\bf d}^*)}u_{A'\oplus P[1]}\star u_D).
\end{split}\end{equation}
Substituting (\ref{nm'1}) into (\ref{zjjg2}), we obtain that
\begin{equation}\label{zhds2}
\begin{split}
&(q^{[M,N]^1}-1)u_{N}\star u_{M'}=q^{\frac{1}{2}\Lambda({{\bf n}^*},{{\bf m}'}^*)}(\sum\limits_{L\ncong M'\oplus N}{|\Ext_{\widetilde{\A}}^1(M',N)_L|}u_L+\\&q^{\frac{1}{2}\lr{{\bf m}',{\bf n}}+[M,N]^1}\sum\limits_{\begin{smallmatrix}D,A',I\\D\ncong N\end{smallmatrix}}|_D\Hom_{\widetilde{\A}}(N,\tau M')_{\tau A'\oplus I}|q^{\frac{1}{2}\Lambda((\tau {\bf a}'+{\bf i})^*,{\bf d}^*)}u_{A'\oplus P[1]}\star u_D).
\end{split}\end{equation}

By (\ref{Hallcheng}), we have that
$u_{P'}\star u_{M'}=q^{\frac{1}{2}\Lambda({{\bf p}'}^*,{{\bf m}'}^*)}u_M$
and $u_{P'}\star u_{L}=q^{\frac{1}{2}\Lambda({{\bf p}'}^*,{{\bf l}}^*)}u_{P'\oplus L}$. It is also easy to see that
$u_{P'}\star u_{A'\oplus P[1]}=q^{\frac{1}{2}\Lambda({{\bf p}'}^*,({\bf a}'-{\bf p})^*)}u_{P'\oplus A'\oplus P[1]}$.
Left multiplying $u_{P'}$ on both sides of the equation (\ref{zhds1}) and substituting these equations above, we obtain that
\begin{equation}\label{zhds3}
\begin{split}
&(q^{[M,N]^1}-1)q^{\frac{1}{2}\Lambda({{\bf p}'}^*,{{\bf m}'}^*)}u_{M}\star u_{N}\\&=q^{\frac{1}{2}\Lambda({{\bf m}'}^*,{\bf n}^*)}(\sum\limits_{L\ncong M'\oplus N}{|\Ext_{\widetilde{\A}}^1(M',N)_L|}q^{\frac{1}{2}\Lambda({{\bf p}'}^*,{{\bf m}'}^*+{{\bf n}}^*)}u_{P'\oplus L}+\\&q^{\frac{1}{2}\lr{{\bf m}',{\bf n}}}\sum\limits_{\begin{smallmatrix}D,A',I\\D\ncong N\end{smallmatrix}}|_D\Hom_{\widetilde{\A}}(N,\tau M')_{\tau A'\oplus I}|q^{\frac{1}{2}\Lambda((\tau {\bf a}'+{\bf i})^*,{\bf d}^*)}q^{\frac{1}{2}\Lambda({{\bf p}'}^*,({\bf a}'-{\bf p})^*)}u_{P'\oplus A'\oplus P[1]}\star u_D).
\end{split}\end{equation}
For each $L$ in the first sum in (\ref{zhds3}), we set $E=L\oplus P'$. It is easy to see that $|\Ext_{\widetilde{\A}}^1(M',N)_L|=|\Ext_{\widetilde{\A}}^1(M,N)_E|$. For each $A'$ in the second sum in (\ref{zhds3}), we set $A=A'\oplus P'$. Then
\begin{equation}
\begin{split}&u_{P'\oplus A'\oplus P[1]}\star u_D=u_{A\oplus P[1]}\star u_D=q^{\frac{1}{2}\Lambda({\bf a}^*,{\bf p}^*)}u_A\star u_{P[1]}\star u_D\\
&=q^{\frac{1}{2}\Lambda({\bf a}^*,{\bf p}^*)}u_A\star u_{I[-1]}\star u_D=q^{\frac{1}{2}\Lambda({\bf a}^*,{\bf p}^*)-\frac{1}{2}\Lambda({\bf i}^*,{\bf d}^*)}u_A\star u_{D\oplus I[-1]}.
\end{split}
\end{equation}
Note that \begin{equation}
\begin{split}
&\Lambda((\tau{{\bf a}'}+{\bf i})^\ast,{\bf d}^*)=\Lambda({^\ast(\tau{{\bf a}'}+{\bf i})},{^\ast{\bf d}})=\Lambda(-{{\bf a}'}^*+{\bf p}^*,{^\ast{\bf d}})\\
&=\Lambda(({\bf p}-{\bf a}')^*,{\bf d}^*)-\Lambda({E}'(\widetilde{Q})({\bf p}-{\bf a}'),B(\widetilde{Q}){\bf d})\\
&=\Lambda(({\bf p}-{\bf a}')^*,{\bf d}^*)+\lr{{\bf p}-{\bf a}',{\bf d}}.
\end{split}
\end{equation}
Hence, we rewrite the equation (\ref{zhds3}) as
\begin{equation}\label{dsdier}
\begin{split}
&(q^{[M,N]^1}-1)u_{M}\star u_{N}=q^{\frac{1}{2}\Lambda({{\bf m}}^*,{\bf n}^*)}\sum\limits_{E\ncong M\oplus N}{|\Ext_{\widetilde{\A}}^1(M,N)_E|}u_{E}+\\&q^{\frac{1}{2}\Lambda({{\bf m}'}^*,{\bf n}^*)-\frac{1}{2}\Lambda({{\bf p}'}^*,{{\bf m}'}^*)+\frac{1}{2}\lr{{\bf m}',{\bf n}}}\sum\limits_{\begin{smallmatrix}D,A',I\\D\ncong N\end{smallmatrix}}|_D\Hom_{\widetilde{\A}}(N,\tau M')_{\tau A'\oplus I}|q^{\frac{1}{2}x}u_{A}\star u_{D\oplus I[-1]},
\end{split}\end{equation}
where
$$x=\Lambda(({\bf p}-{\bf a}')^*,{\bf d}^*)+\lr{{\bf p}-{\bf a}',{\bf d}}+\Lambda({{\bf p}'}^*,({\bf a}'-{\bf p})^*)+\Lambda({\bf a}^*,{\bf p}^*)-\Lambda({\bf i}^*,{\bf d}^*).$$
Since $\Lambda({\bf i}^*,{\bf d}^*)=\Lambda({^*{\bf i}}+B(\widetilde{Q}){\bf i},{\bf d}^*)=\Lambda({\bf p}^*,{\bf d}^*)+\lr{{\bf d},{\bf i}}=\Lambda({\bf p}^*,{\bf d}^*)+\lr{{\bf p},{\bf d}}$,
we obtain that
\begin{equation*}
\begin{split}x&=-\Lambda({{\bf a}'}^*,{\bf d}^*)-\lr{{\bf a}',{\bf d}}+\Lambda({{\bf p}'}^*,({\bf a}'-{\bf p})^*)+\Lambda({({\bf a}'+{\bf p}')}^*,{\bf p}^*)
\\
&=\Lambda(({\bf p}'-{\bf p}+{\bf d})^*,{{\bf a}'}^*)-\lr{{\bf a}',{\bf d}}.
\end{split}
\end{equation*}
By the exact sequence (\ref{zhxl}), we have that ${\bf d}-{\bf n}+\tau{\bf m}'-\tau{\bf a}'-{\bf i}=\mathbf{0}$, and then
\begin{equation*}
\begin{split}
\mathbf{0}={^*({\bf d}-{\bf n}+\tau{\bf m}'-\tau{\bf a}'-{\bf i})}=
({\bf d}-{\bf n})^*-B(\widetilde{Q})({\bf d}-{\bf n})-({\bf m}'-{\bf a}')^*-{\bf p}^*.
\end{split}
\end{equation*}
So $({\bf p}'-{\bf p}+{\bf d})^*=({\bf p}'+{\bf m}'-{\bf a}'+{\bf n})^*+B(\widetilde{Q})({\bf d}-{\bf n})=({\bf p}'-{\bf a}+{\bf m}+{\bf n})^*+B(\widetilde{Q})({\bf d}-{\bf n})$. Hence,
\begin{equation*}
\begin{split}x&=\Lambda(({\bf p}'-{\bf a}+{\bf m}+{\bf n})^*,{{\bf a}'}^*)+\Lambda(B(\widetilde{Q})({\bf d}-{\bf n}),{{\bf a}'}^*)-\lr{{\bf a}',{\bf d}}\\
&=\Lambda(({\bf p}'-{\bf a}+{\bf m}+{\bf n})^*,{{\bf a}'}^*)+\lr{{\bf a}',{\bf d}-{\bf n}}-\lr{{\bf a}',{\bf d}}\\
&=\Lambda({{\bf p}'}^*,{{\bf a}'}^*)+\Lambda(({\bf m}-{\bf a})^*,{{\bf a}'}^*)-\Lambda({{\bf a}'}^*,{\bf n}^*)-\lr{{\bf a}',{\bf n}}.
\end{split}
\end{equation*}
Thus, the equation (\ref{dsdier}) is rewritten as
\begin{equation}
\begin{split}
&(q^{[M,N]^1}-1)u_M\star u_N=q^{\frac{1}{2}\Lambda({\bf m}^*,{\bf n}^*)}\sum\limits_{E\ncong M\oplus N}
|\Ext_{\widetilde{\A}}^1(M,N)_E|u_E+\\&\sum\limits_{\begin{smallmatrix}D,A',I\\D\ncong N\end{smallmatrix}}q^{\frac{1}{2}\Lambda(({\bf m}-{{\bf a}})^*,{({\bf n}+{\bf a})}^*)+\frac{1}{2}\lr{{{\bf m}}-{{\bf a}},{\bf n}}}|_D\Hom_{\widetilde{\A}}(N,\tau M')_{\tau A'\oplus I}|u_{A}\star u_{D\oplus I[-1]}.
\end{split}\end{equation}

For each short exact sequence $0\longrightarrow B\longrightarrow M'\longrightarrow A'\longrightarrow 0$, it is easy to see that $B$ has no nonzero projective direct summands, and $\Hom_{\widetilde{\A}}(B,P')=0$. Then $F_{A'B}^{M'}=F_{AB}^M$. Using \cite[Lemma 4.4]{Rupel2}, we obtain that
$|_D\Hom_{\widetilde{\A}}(N,\tau M')_{\tau A'\oplus I}|=|_D\Hom_{\widetilde{\A}}(N,\tau M)_{\tau A\oplus I}|.$
So we finish the proof of the first equation.

Now, we prove the second equation. Since $\Ext_{\widetilde{\A}}^1(N,M)=0$ implies $\Ext_{\widetilde{\A}}^1(N,P')=0$, we have that
$u_N\star u_{P'}=q^{\frac{1}{2}\Lambda({\bf n}^*,{{\bf p}'}^*)}u_{N\oplus P'}$. Since $u_{P'}\star u_{N}=q^{\frac{1}{2}\Lambda({{\bf p}'}^*,{\bf n}^*)}u_{N\oplus P'}$, we obtain that $u_{P'}\star u_{N}=q^{\Lambda({{\bf p}'}^*,{\bf n}^*)}u_N\star u_{P'}$. Then \begin{equation}
\begin{split}u_{P'}\star u_{N}\star u_{M'}&=q^{\Lambda({{\bf p}'}^*,{\bf n}^*)}u_N\star u_{P'}\star u_{M'}\\&=q^{\Lambda({{\bf p}'}^*,{\bf n}^*)+\frac{1}{2}\Lambda({{\bf p}'}^*,{{\bf m}'}^*)}u_N\star u_M\\&=q^{\Lambda({{\bf p}'}^*,{\bf n}^*)+\frac{1}{2}\Lambda({{\bf p}'}^*,{{\bf m}'}^*)+\frac{1}{2}\Lambda({\bf n}^*,{\bf m}^*)}u_{M\oplus N}.\end{split}\end{equation}
Left multiplying $u_{P'}$ on both sides of the equation (\ref{zhds2}), we obtain that
\begin{equation}
\begin{split}
&(q^{[M,N]^1}-1)q^{\Lambda({{\bf p}'}^*,{\bf n}^*)+\frac{1}{2}\Lambda({{\bf p}'}^*,{{\bf m}'}^*)+\frac{1}{2}\Lambda({\bf n}^*,{\bf m}^*)}u_{M\oplus N}\\&=q^{\frac{1}{2}\Lambda({\bf n}^*,{{\bf m}'}^*)}(\sum\limits_{L\ncong M'\oplus N}{|\Ext_{\widetilde{\A}}^1(M',N)_L|}q^{\frac{1}{2}\Lambda({{\bf p}'}^*,{{\bf m}'}^*+{{\bf n}}^*)}u_{P'\oplus L}+\\&q^{\frac{1}{2}\lr{{\bf m}',{\bf n}}+[M,N]^1}\sum\limits_{\begin{smallmatrix}D,A',I\\D\ncong N\end{smallmatrix}}|_D\Hom_{\widetilde{\A}}(N,\tau M')_{\tau A'\oplus I}|q^{\frac{1}{2}\Lambda((\tau {\bf a}'+{\bf i})^*,{\bf d}^*)}q^{\frac{1}{2}\Lambda({{\bf p}'}^*,({\bf a}'-{\bf p})^*)}u_{P'\oplus A'\oplus P[1]}\star u_D).
\end{split}\end{equation}
Noting that $$\Lambda({{\bf p}'}^*,{\bf n}^*)+\frac{1}{2}\Lambda({{\bf p}'}^*,{{\bf m}'}^*)+\frac{1}{2}\Lambda({\bf n}^*,{\bf m}^*)=\frac{1}{2}\Lambda({\bf n}^*,{{\bf m}'}^*)+\frac{1}{2}\Lambda({{\bf p}'}^*,{{\bf m}'}^*)+\frac{1}{2}\Lambda({{\bf p}'}^*,{{\bf n}}^*),$$
we get that
\begin{equation}
\begin{split}
&(q^{[M,N]^1}-1)u_{M\oplus N}=\\&\sum\limits_{E\ncong M\oplus N}
|\Ext_{\widetilde{\A}}^1(M,N)_E|u_E+\sum\limits_{\begin{smallmatrix}D,A',I\\D\ncong N\end{smallmatrix}}q^y|_D\Hom_{\widetilde{\A}}(N,\tau M')_{\tau A'\oplus I}|u_{A}\star u_{D\oplus I[-1]}
\end{split}\end{equation}
where $A=A'\oplus P'$ and \begin{equation}
\begin{split}y&=\frac{1}{2}\lr{{\bf m}',{\bf n}}+[M,N]^1+\frac{1}{2}\Lambda((\tau {\bf a}'+{\bf i})^*,{\bf d}^*)+\frac{1}{2}\Lambda({{\bf p}'}^*,({\bf a}'-{\bf p})^*)\\&\quad\quad+\frac{1}{2}\Lambda({\bf a}^*,{\bf p}^*)-\frac{1}{2}\Lambda({\bf i}^*,{\bf d}^*)-\frac{1}{2}\Lambda({{\bf p}'}^*,{\bf n}^*)-\frac{1}{2}\Lambda({{\bf p}'}^*,{{\bf m}'}^*)\\&=\frac{1}{2}\lr{{\bf m}',{\bf n}}+[M,N]^1+\frac{1}{2}x-\frac{1}{2}\Lambda({{\bf p}'}^*,{\bf n}^*)-\frac{1}{2}\Lambda({{\bf p}'}^*,{{\bf m}'}^*)\\&=\frac{1}{2}\lr{{\bf m}',{\bf n}}+[M,N]^1+\frac{1}{2}\Lambda(({\bf p}'-{\bf a}+{\bf m}+{\bf n})^*,{{\bf a}'}^*)\\&\quad\quad-\frac{1}{2}\lr{{\bf a}',{\bf n}}-\frac{1}{2}\Lambda({{\bf p}'}^*,{\bf n}^*)-\frac{1}{2}\Lambda({{\bf p}'}^*,{{\bf m}'}^*)\\&=\frac{1}{2}\Lambda(({\bf m}+{\bf n}-{{\bf a}})^*,{{\bf a}}^*)+\frac{1}{2}\lr{{{\bf m}}-{{\bf a}},{\bf n}}+[M,N]^1.\end{split}\end{equation}
Therefore, we complete the proof.
\end{proof}

We fail to establish a homomorphism of algebras from $\mathcal {D}\mathcal {H}_\Lambda^{{cl}_2}(\widetilde{\A})$ to the quantum torus $\mathcal{T}_\Lambda$. That is, we can not immediately get the corresponding equation of Theorem \ref{ddl} in the quantum cluster algebra. In the next section, we will directly prove that the desired equation holds in $\mathcal{T}_\Lambda$ via the quantum cluster algebra approach. What's more, for the quantum cluster algebra case, we do not need to assume the condition that $\Ext_{\widetilde{\A}}^1(N,M)=0$ as in Theorem \ref{ddl}.

\section{Cluster multiplication theorem, $\textrm{II}$}
In this section, we prove the cluster multiplication theorem for the quantum cluster characters $X_M$ and $X_N$ in the quantum cluster algebras, which are of the same form as (\ref{halltimes}) in Theorem \ref{ddl}.

For the simplicity of notation, we set \begin{equation*}[M,N]^0:=\dim_k\Hom_{\widetilde{\A}}(M,N)~\text{and}~\varepsilon_{MN}^E:=|\Ext^1_{\widetilde{\A}}(M,N)_E|\end{equation*}
for any $M,N,E\in\widetilde{\A}$.

First of all, let us collect the following lemmas for later use.

\begin{lemma} $($Green's formula \cite{Gr95}$)$
For any $M,N,X,Y\in\widetilde{\A}$, we have that
\begin{equation}
\sum\limits_{E}\varepsilon_{MN}^EF_{XY}^E=\sum\limits_{A,B,C,D}q^{[M,N]^0-[A,C]^0-[B,D]^0-\lr{{\bf a},{\bf d}}}F_{AB}^MF_{CD}^N\varepsilon_{AC}^X\varepsilon_{BD}^Y.
\end{equation}
\end{lemma}
\begin{lemma}
For any ${\bf a},{\bf b},{\bf c},{\bf d}\in\mathbb{Z}^m$,
we have that $$\Lambda(-{\bf b}^*-{^*{\bf a}},-{\bf d}^*-{^*{\bf c}})=\Lambda(({\bf a}+{\bf b})^*,({\bf c}+{\bf d})^*)+\lr{{\bf b},{\bf c}}-\lr{{\bf d},{\bf a}}.$$
\end{lemma}
\begin{proof}
It is easily proved by using Lemma \ref{sjishu}.
\end{proof}
\begin{lemma}\label{zscx}
Let $M,N\in\widetilde{\A}$. Keep the notations in (\ref{zhxl}). Set $A=A'\oplus P'$, where $P'$ is the maximal projective direct summand of $M$. Then we have that
\begin{flalign*}&\Lambda(({\bf m}-{{\bf a}})^*,{({\bf n}+{\bf a})}^*)+\lr{{{\bf m}}-{{\bf a}},{\bf n}}
\\&=\Lambda({\bf m}^*,{\bf n}^*)+\lr{{\bf m},{\bf n}}-\lr{{\bf a},{\bf d}}-\Lambda({\bf a}^*,({\bf d}-{\bf i})^*)+\lr{{\bf a},{\bf i}}.
\end{flalign*}
\end{lemma}
\begin{proof}
By the exact sequence (\ref{zhxl}), we have that ${\bf d}-{\bf i}={\bf n}-\tau({\bf m}'-{\bf a}')={\bf n}-\tau({\bf m}-{\bf a})$, and then
${^*({\bf d}-{\bf i})}={^*{\bf n}}-{^*(\tau({\bf m}-{\bf a}))}={^*{\bf n}}+({\bf m}-{\bf a})^*$. Thus, \begin{equation}
\begin{split}\Lambda({\bf a}^*,({\bf d}-{\bf i})^*)&=\Lambda({^*{\bf a}},{^*(\bf{d}-{\bf i})})\\&=\Lambda({^*{\bf a}},{^*{\bf n}})+\Lambda({^*{\bf a}},({\bf m}-{\bf a})^*)\\&=\Lambda({{\bf a}^*},{{\bf n}^*})+\Lambda({{\bf a}^*},({\bf m}-{\bf a})^*)-\Lambda({B(\widetilde{Q})}{\bf a},({\bf m}-{\bf a})^*)\\&=\Lambda({{\bf a}^*},{{\bf n}^*})+\Lambda({{\bf a}^*},({\bf m}-{\bf a})^*)-\lr{{\bf m}-{\bf a},{\bf a}}.\end{split}\end{equation}
Hence,
\begin{equation}
\begin{split}&\Lambda(({\bf m}-{{\bf a}})^*,{({\bf n}+{\bf a})}^*)+\lr{{{\bf m}}-{{\bf a}},{\bf n}}+\Lambda({\bf a}^*,({\bf d}-{\bf i})^*)\\&=\Lambda({\bf m}^*,{\bf n}^*)+\lr{{\bf m}-{\bf a},{\bf n}-{\bf a}}\\
&=\Lambda({\bf m}^*,{\bf n}^*)+\lr{{\bf m},{\bf n}}-\lr{{\bf m}-{\bf a},{\bf a}}-\lr{{\bf a},{\bf n}}\\
&=\Lambda({\bf m}^*,{\bf n}^*)+\lr{{\bf m},{\bf n}}+\lr{{\bf a},\tau({\bf m}-{\bf a})}-\lr{{\bf a},{\bf n}}\\
&=\Lambda({\bf m}^*,{\bf n}^*)+\lr{{\bf m},{\bf n}}+\lr{{\bf a},{\bf n}-{\bf d}+{\bf i}}-\lr{{\bf a},{\bf n}}\\
&=\Lambda({\bf m}^*,{\bf n}^*)+\lr{{\bf m},{\bf n}}-\lr{{\bf a},{\bf d}}+\lr{{\bf a},{\bf i}}.\end{split}\end{equation}
Therefore, we finish the proof.
\end{proof}
\begin{theorem}\label{ddlz}
Let $M,N\in\widetilde{\A}$. Then we have the following equation in $\mathcal{T}_\Lambda:$
\begin{equation}
\begin{split}
&(q^{[M,N]^1}-1)X_M X_N=q^{\frac{1}{2}\Lambda({\bf m}^*,{\bf n}^*)}\sum\limits_{E\ncong M\oplus N}
|\Ext_{\widetilde{\A}}^1(M,N)_E|X_E+\\&\sum\limits_{\begin{smallmatrix}D,A,I\\D\ncong N\end{smallmatrix}}q^{\frac{1}{2}\Lambda(({\bf m}-{{\bf a}})^*,{({\bf n}+{\bf a})}^*)+\frac{1}{2}\lr{{{\bf m}}-{{\bf a}},{\bf n}}}|_D\Hom_{\widetilde{\A}}(N,\tau M)_{\tau A\oplus I}|X_{A} X_{D\oplus I[-1]},
\end{split}\end{equation}
where each $A$ has the same maximal projective direct summand as $M$.
\end{theorem}
\begin{proof}
%We give the proof in a similar way as \cite[Theorem 4.5]{Rupel2}.

Recall from (\ref{qcltz}) that
\begin{flalign*}X_M=\sum\limits_{\mathbf{e}}v^{-\lr{\mathbf{e},\mathbf{m}-\mathbf{e}}}|\mathrm{Gr}_{\mathbf{e}}M|
X^{-\mathbf{e}^\ast-^\ast(\mathbf{m}-\mathbf{e})}=
\sum\limits_{A,B}q^{-\frac{1}{2}\lr{\mathbf{b},\mathbf{a}}}F_{AB}^M
X^{-\mathbf{b}^\ast-^\ast\mathbf{a}}.\end{flalign*}
Then
\begin{flalign*}
X_MX_N&=\sum\limits_{A,B}q^{-\frac{1}{2}\lr{\mathbf{b},\mathbf{a}}}F_{AB}^M
X^{-\mathbf{b}^\ast-^\ast\mathbf{a}}\sum\limits_{C,D}q^{-\frac{1}{2}\lr{\mathbf{d},\mathbf{c}}}F_{CD}^N
X^{-\mathbf{d}^\ast-^\ast\mathbf{c}}\\
&=\sum\limits_{A,B,C,D}q^{-\frac{1}{2}\lr{\mathbf{b},\mathbf{a}}-\frac{1}{2}\lr{\mathbf{d},\mathbf{c}}+\frac{1}{2}\Lambda(-\mathbf{b}^\ast-^\ast\mathbf{a},-\mathbf{d}^\ast-^\ast\mathbf{c})}
F_{AB}^MF_{CD}^NX^{-({\bf b}+{\bf d})^\ast-^\ast({\bf a}+{\bf c})}\\
&=q^{\frac{1}{2}\Lambda({\bf m}^*,{\bf n}^*)}\sum\limits_{A,B,C,D}q^{\lr{{\bf b},{\bf c}}-\frac{1}{2}\lr{{\bf b}+{\bf d},{\bf a}+{\bf c}}}F_{AB}^MF_{CD}^NX^{-({\bf b}+{\bf d})^\ast-^\ast({\bf a}+{\bf c})}.
\end{flalign*}
Using Green's formula, we have that
\begin{flalign*}
&\sum\limits_{E}\varepsilon_{MN}^EX_E=\sum\limits_{E,X,Y}\varepsilon_{MN}^Eq^{-\frac{1}{2}\lr{{\bf y},{\bf x}}}F_{XY}^EX^{-{\bf y}^*-^*{\bf x}}\\
&=\sum\limits_{A,B,C,D,X,Y}q^{[M,N]^0-[A,C]^0-[B,D]^0-\lr{{\bf a},{\bf d}}}q^{-\frac{1}{2}\lr{{\bf b}+{\bf d},{\bf a}+{\bf c}}}F_{AB}^MF_{CD}^N\varepsilon_{AC}^X\varepsilon_{BD}^YX^{-({\bf b}+{\bf d})^\ast-^\ast({\bf a}+{\bf c})}\\
&=\sum\limits_{A,B,C,D}q^{[M,N]^1}q^{\lr{{\bf b},{\bf c}}}q^{-\frac{1}{2}\lr{{\bf b}+{\bf d},{\bf a}+{\bf c}}}F_{AB}^MF_{CD}^NX^{-({\bf b}+{\bf d})^\ast-^\ast({\bf a}+{\bf c})}.
\end{flalign*}

Using an equation on Hall numbers (cf. \cite{Hubery1}, \cite[Lemma 4.11]{Rupel2}) \begin{equation}
\sum\limits_{X,Y}F_{XY}^{M\oplus N}=\sum\limits_{A,B,C,D}q^{[B,C]^0}F_{AB}^MF_{CD}^N,
\end{equation}
we obtain that
\begin{equation}
X_{M\oplus N}=\sum\limits_{A,B,C,D}q^{-\frac{1}{2}\lr{{\bf b}+{\bf d},{\bf a}+{\bf c}}}q^{[B,C]^0}F_{AB}^MF_{CD}^NX^{-({\bf b}+{\bf d})^\ast-^\ast({\bf a}+{\bf c})}.
\end{equation}
Set $$\sigma_1:=q^{\frac{1}{2}\Lambda({\bf m}^*,{\bf n}^*)}\sum\limits_{E\ncong M\oplus N}
\frac{|\Ext_{\widetilde{\A}}^1(M,N)_E|}{q^{[M,N]^1}-1}X_E.$$ Then
\begin{flalign*}
\sigma_1&=
q^{\frac{1}{2}\Lambda({\bf m}^*,{\bf n}^*)}\frac{\sum\limits_{E}|\Ext_{\widetilde{\A}}^1(M,N)_E|X_E-X_{M\oplus N}}{q^{[M,N]^1}-1}\\
&=q^{\frac{1}{2}\Lambda({\bf m}^*,{\bf n}^*)}\sum\limits_{A,B,C,D}\frac{q^{[M,N]^1}-q^{[B,C]^1}}{q^{[M,N]^1}-1}q^{\lr{{\bf b},{\bf c}}}q^{-\frac{1}{2}\lr{{\bf b}+{\bf d},{\bf a}+{\bf c}}}F_{AB}^MF_{CD}^NX^{-({\bf b}+{\bf d})^\ast-^\ast({\bf a}+{\bf c})}.
\end{flalign*}

Set $$\sigma_2:=\sum\limits_{\begin{smallmatrix}D,A,I\\D\ncong N\end{smallmatrix}}q^{\frac{1}{2}\Lambda(({\bf m}-{{\bf a}})^*,{({\bf n}+{\bf a})}^*)+\frac{1}{2}\lr{{{\bf m}}-{{\bf a}},{\bf n}}}\frac{|_D\Hom_{\widetilde{\A}}(N,\tau M)_{\tau A\oplus I}|}{q^{[M,N]^1}-1}X_{A} X_{D\oplus I[-1]}.$$
Then
\begin{flalign*}
\sigma_2&=\sum\limits_{\begin{smallmatrix}D,A,I,K,L,X,Y\\D\ncong N\end{smallmatrix}}q^{\frac{1}{2}\Lambda(({\bf m}-{{\bf a}})^*,{({\bf n}+{\bf a})}^*)+\frac{1}{2}\lr{{{\bf m}}-{{\bf a}},{\bf n}}}\frac{|_D\Hom_{\widetilde{\A}}(N,\tau M)_{\tau A\oplus I}|}{q^{[M,N]^1}-1}\\&\quad\quad\quad\quad\quad\quad\times q^{-\frac{1}{2}\lr{{\bf l},{\bf k}}}F_{KL}^AX^{-{\bf l}^*-^*{\bf k}}q^{-\frac{1}{2}\lr{{\bf y},{\bf x}-{\bf i}}}F_{XY}^DX^{-{\bf y}^*-^*({\bf x}-{\bf i})}\\
&=\sum\limits_{\begin{smallmatrix}D,A,I,K,L,X,Y\\D\ncong N\end{smallmatrix}}q^{\frac{1}{2}\Lambda(({\bf m}-{{\bf a}})^*,{({\bf n}+{\bf a})}^*)+\frac{1}{2}\lr{{{\bf m}}-{{\bf a}},{\bf n}}}\frac{|_D\Hom_{\widetilde{\A}}(N,\tau M)_{\tau A\oplus I}|}{q^{[M,N]^1}-1}\\&\quad\quad\quad\times q^{-\frac{1}{2}\lr{{\bf l},{\bf k}}-\frac{1}{2}\lr{{\bf y},{\bf x}-{\bf i}}+\frac{1}{2}\Lambda(-{\bf l}^*-^*{\bf k},-{\bf y}^*-^*{\bf x}+^*{\bf i}))}F_{KL}^AF_{XY}^DX^{-({\bf l}+{\bf y})^*-^*({\bf k}+{\bf x})+^*{\bf i}}.
\end{flalign*}
By \cite[Lemma 4.4]{Rupel2},
\begin{flalign*}
\sigma_2&=\sum\limits_{\begin{smallmatrix}A,B,C,D,I,K,L,X,Y\\D\ncong N\end{smallmatrix}}q^{\frac{1}{2}\Lambda(({\bf m}-{{\bf a}})^*,{({\bf n}+{\bf a})}^*)+\frac{1}{2}\lr{{{\bf m}}-{{\bf a}},{\bf n}}}\frac{a_CF_{CD}^NF_{AB}^{M}F_{IC}^{\tau B}}{q^{[M,N]^1}-1}\\&\quad\quad\quad\times q^{-\frac{1}{2}\lr{{\bf l},{\bf k}}-\frac{1}{2}\lr{{\bf y},{\bf x}-{\bf i}}+\frac{1}{2}\Lambda(-{\bf l}^*-^*{\bf k},-{\bf y}^*-^*{\bf x}+^*{\bf i}))}F_{KL}^AF_{XY}^DX^{-({\bf l}+{\bf y})^*-^*({\bf k}+{\bf x})+^*{\bf i}}.
\end{flalign*}
We remind that each $B$ has no nonzero projective direct summands in the above sum.

Since ${^*{\bf i}}={^*\tau{\bf b}}-{^*{\bf c}}=-{\bf b}^*-{^*{\bf c}}$, we obtain that
\begin{flalign*}
\sigma_2&=\sum\limits_{\begin{smallmatrix}A,B,C,D,I,K,L,X,Y\\D\ncong N\end{smallmatrix}}q^{\frac{1}{2}\Lambda(({\bf m}-{{\bf a}})^*,{({\bf n}+{\bf a})}^*)+\frac{1}{2}\lr{{{\bf m}}-{{\bf a}},{\bf n}}}\frac{a_CF_{CD}^NF_{AB}^{M}F_{IC}^{\tau B}}{q^{[M,N]^1}-1}\\&\quad\quad\quad\times q^{-\frac{1}{2}\lr{{\bf l},{\bf k}}-\frac{1}{2}\lr{{\bf y},{\bf x}-{\bf i}}+\frac{1}{2}\Lambda(-{\bf l}^*-^*{\bf k},-{\bf y}^*-^*{\bf x}+^*{\bf i}))}F_{KL}^AF_{XY}^DX^{-({\bf l}+{\bf y}+{\bf b})^*-^*({\bf k}+{\bf x}+{\bf c})}.
\end{flalign*}

In what follows, we focus on the calculations of the exponents in $\sigma_2$:
\begin{flalign*}
&t:=\frac{1}{2}\Lambda(({\bf m}-{{\bf a}})^*,{({\bf n}+{\bf a})}^*)+\frac{1}{2}\lr{{{\bf m}}-{{\bf a}},{\bf n}}-\frac{1}{2}\lr{{\bf l},{\bf k}}-\frac{1}{2}\lr{{\bf y},{\bf x}-{\bf i}}\\&\quad\quad+\frac{1}{2}\Lambda(-{\bf l}^*-^*{\bf k},-{\bf y}^*-^*{\bf x})-\frac{1}{2}\Lambda({\bf l}^*+^*{\bf k},^*{\bf i})\\
&=\frac{1}{2}\Lambda({\bf m}^*,{\bf n}^*)+\frac{1}{2}\lr{{\bf m},{\bf n}}-\frac{1}{2}\lr{{\bf a},{\bf d}}-\frac{1}{2}\Lambda({\bf a}^*,({\bf d}-{\bf i})^*)+\frac{1}{2}\lr{{\bf a},{\bf i}}-\frac{1}{2}\lr{{\bf l},{\bf k}}\\
&\quad-\frac{1}{2}\lr{{\bf y},{\bf x}-{\bf i}}+\frac{1}{2}\Lambda({\bf a}^*,{\bf d}^*)-\frac{1}{2}\lr{{\bf y},{\bf a}-{\bf l}}+\frac{1}{2}\lr{{\bf l},{\bf d}-{\bf y}}
-\frac{1}{2}\Lambda({\bf l}^*+^*{\bf k},^*{\bf i})\\
&=\frac{1}{2}\Lambda({\bf m}^*,{\bf n}^*)+\frac{1}{2}\lr{{\bf m},{\bf n}}-\frac{1}{2}\lr{{\bf a},{\bf d}}+\frac{1}{2}\Lambda({\bf a}^*,{\bf i}^*)+\frac{1}{2}\lr{{\bf a},{\bf i}}-\frac{1}{2}\lr{{\bf l},{\bf k}}-\frac{1}{2}\lr{{\bf y},{\bf x}-{\bf i}}\\&\quad-\frac{1}{2}\lr{{\bf y},{\bf a}-{\bf l}}+\frac{1}{2}\lr{{\bf l},{\bf d}-{\bf y}}
-\frac{1}{2}\Lambda({\bf l}^*+^*{\bf k},{\bf i}^*)+\frac{1}{2}\Lambda({\bf l}^*+^*{\bf k},B(\widetilde{Q}){\bf i}).
\end{flalign*}
Noting that \begin{flalign*}&\frac{1}{2}\Lambda({\bf a}^*,{\bf i}^*)-\frac{1}{2}\Lambda({\bf l}^*+^*{\bf k},{\bf i}^*)=\frac{1}{2}\Lambda(({\bf a}-{\bf l})^*-^*{\bf k},{\bf i}^*)\\&=
\frac{1}{2}\Lambda({\bf k}^*-^*{\bf k},{\bf i}^*)=\frac{1}{2}\Lambda(B(\widetilde{Q}){\bf k},{\bf i}^*)=\frac{1}{2}\lr{{\bf i},{\bf k}}\end{flalign*}
and \begin{flalign*}\frac{1}{2}\Lambda({\bf l}^*+^*{\bf k},B(\widetilde{Q}){\bf i})=-\frac{1}{2}\lr{{\bf l},{\bf i}}-\frac{1}{2}\lr{{\bf i},{\bf k}},\end{flalign*}
we obtain that
\begin{flalign*}
t&=\frac{1}{2}\Lambda({\bf m}^*,{\bf n}^*)+\frac{1}{2}\lr{{\bf m},{\bf n}}-\frac{1}{2}\lr{{\bf a},{\bf d}}+\frac{1}{2}\lr{{\bf a},{\bf i}}-\frac{1}{2}\lr{{\bf l},{\bf k}}-\frac{1}{2}\lr{{\bf y},{\bf x}}\\&\quad+\frac{1}{2}\lr{{\bf y},{\bf i}}-\frac{1}{2}\lr{{\bf y},{\bf a}}+\frac{1}{2}\lr{{\bf y},{\bf l}}+\frac{1}{2}\lr{{\bf l},{\bf d}}
-\frac{1}{2}\lr{{\bf l},{\bf y}}-\frac{1}{2}\lr{{\bf l},{\bf i}}\\
&=\frac{1}{2}\Lambda({\bf m}^*,{\bf n}^*)+\frac{1}{2}\lr{{\bf m},{\bf n}}-\frac{1}{2}\lr{{\bf a},{\bf d}}+\frac{1}{2}\lr{{\bf y},{\bf l}-{\bf a}+{\bf i}-{\bf x}}\\&\quad+\frac{1}{2}\lr{{\bf l},{\bf d}-{\bf y}-{\bf i}-{\bf k}}+\frac{1}{2}\lr{{\bf a},{\bf i}}.
\end{flalign*}
Since ${\bf l}-{\bf a}=-{\bf k}$ and ${\bf d}-{\bf y}={\bf x}$, we obtain that
\begin{flalign*}
t&=\frac{1}{2}\Lambda({\bf m}^*,{\bf n}^*)+\frac{1}{2}\lr{{\bf m},{\bf n}}-\frac{1}{2}\lr{{\bf a},{\bf d}}+\frac{1}{2}\lr{{\bf a},{\bf i}}-\frac{1}{2}\lr{{\bf y},{\bf x}-{\bf i}}\\&\quad-\frac{1}{2}\lr{{\bf y},{\bf k}}+\frac{1}{2}\lr{{\bf l},{\bf x}-{\bf i}}-\frac{1}{2}\lr{{\bf l},{\bf k}}\\
&=\frac{1}{2}\Lambda({\bf m}^*,{\bf n}^*)+\frac{1}{2}\lr{{\bf m},{\bf n}}-\frac{1}{2}\lr{{\bf a},{\bf d}}+\frac{1}{2}\lr{{\bf a},{\bf i}}+\frac{1}{2}\lr{{\bf l}-{\bf y},{\bf x}-{\bf i}}-\frac{1}{2}\lr{{\bf l}+{\bf y},{\bf k}}.
\end{flalign*}
Thus, \begin{flalign*}&t=t-\frac{1}{2}\lr{{\bf l}+{\bf y},{\bf x}}+\frac{1}{2}\lr{{\bf l}+{\bf y},{\bf x}}\\
&=\frac{1}{2}\Lambda({\bf m}^*,{\bf n}^*)+\frac{1}{2}\lr{{\bf m},{\bf n}}-\frac{1}{2}\lr{{\bf a},{\bf d}}+\frac{1}{2}\lr{{\bf a},{\bf i}}-\frac{1}{2}\lr{{\bf l}+{\bf y},{\bf k}+{\bf x}}+\lr{{\bf l},{\bf x}}-\frac{1}{2}\lr{{\bf l}-{\bf y},{\bf i}}\\
&=\frac{1}{2}\Lambda({\bf m}^*,{\bf n}^*)+\frac{1}{2}\lr{{\bf m},{\bf n}}-\frac{1}{2}\lr{{\bf a},{\bf d}}-\frac{1}{2}\lr{{\bf l}+{\bf y},{\bf k}+{\bf x}}+\lr{{\bf l},{\bf x}}+\frac{1}{2}\lr{{\bf a}-{\bf l}+{\bf y},{\bf i}}\\
&=\frac{1}{2}\Lambda({\bf m}^*,{\bf n}^*)+\frac{1}{2}\lr{{\bf m},{\bf n}}-\frac{1}{2}\lr{{\bf a},{\bf d}}-\frac{1}{2}\lr{{\bf l}+{\bf y},{\bf k}+{\bf x}}+\lr{{\bf l},{\bf x}}+\frac{1}{2}\lr{{\bf k}+{\bf y},{\bf i}}.\end{flalign*}
Observing that \begin{flalign*}&\lr{{\bf k},{\bf n}-{\bf d}}=\lr{{\bf k},{\bf c}}=\lr{{\bf k},\tau{\bf b}-{\bf i}}=\lr{{\bf k},\tau{\bf m}'-\tau{\bf a}'-{\bf i}}
\\&=-\lr{{\bf m}'-{\bf a}',{\bf k}}-\lr{{\bf k},{\bf i}}=-\lr{{\bf m}-{\bf a},{\bf k}}-\lr{{\bf k},{\bf i}},\end{flalign*}
similarly, $$\lr{{\bf y},{\bf n}-{\bf d}}=-\lr{{\bf m}-{\bf a},{\bf y}}-\lr{{\bf y},{\bf i}},$$
we have that
\begin{flalign*}
&\lr{{\bf k},{\bf d}}+\lr{{\bf a},{\bf k}}=\lr{{\bf k},{\bf n}}+\lr{{\bf m},{\bf k}}+\lr{{\bf k},{\bf i}}\\
&\lr{{\bf y},{\bf d}}+\lr{{\bf a},{\bf y}}=\lr{{\bf y},{\bf n}}+\lr{{\bf m},{\bf y}}+\lr{{\bf y},{\bf i}}.
\end{flalign*}
Hence, \begin{flalign*}&-\frac{1}{2}\lr{{\bf a},{\bf d}}-\frac{1}{2}\lr{{\bf l}+{\bf y},{\bf k}+{\bf x}}+\lr{{\bf l},{\bf x}}+\frac{1}{2}\lr{{\bf k}+{\bf y},{\bf i}}\\
&=-\frac{1}{2}\lr{{\bf a},{\bf d}}-\frac{1}{2}\lr{{\bf a}-{\bf k}+{\bf y},{\bf k}+{\bf d}-{\bf y}}+\lr{{\bf a}-{\bf k},{\bf d}-{\bf y}}+\frac{1}{2}\lr{{\bf k}+{\bf y},{\bf i}}\\
&=-\frac{1}{2}(\lr{{\bf y},{\bf d}}+\lr{{\bf a},{\bf y}})-\frac{1}{2}(\lr{{\bf k},{\bf d}}+\lr{{\bf a},{\bf k}})+\frac{1}{2}\lr{{\bf k},{\bf y}}+\frac{1}{2}\lr{{\bf k},{\bf k}}\\&\quad-\frac{1}{2}\lr{{\bf y},{\bf k}}+\frac{1}{2}\lr{{\bf y},{\bf y}}+\frac{1}{2}\lr{{\bf k}+{\bf y},{\bf i}}\\
&=-\frac{1}{2}\lr{{\bf y},{\bf n}}-\frac{1}{2}\lr{{\bf m},{\bf y}}-\frac{1}{2}\lr{{\bf y},{\bf i}}-\frac{1}{2}\lr{{\bf k},{\bf n}}-\frac{1}{2}\lr{{\bf m},{\bf k}}-\frac{1}{2}\lr{{\bf k},{\bf i}}\\
&\quad+\frac{1}{2}\lr{{\bf k}+{\bf y},{\bf i}}+\frac{1}{2}\lr{{\bf k},{\bf y}}+\frac{1}{2}\lr{{\bf k},{\bf k}}-\frac{1}{2}\lr{{\bf y},{\bf k}}+\frac{1}{2}\lr{{\bf y},{\bf y}}\\
&=\lr{{\bf m}-{\bf k},{\bf n}-{\bf y}}-\frac{1}{2}\lr{{\bf m}-{\bf k}+{\bf y},{\bf k}+{\bf n}-{\bf y}}-\frac{1}{2}\lr{{\bf m},{\bf n}},
\end{flalign*}
and then $$t=\frac{1}{2}\Lambda({\bf m}^*,{\bf n}^*)+\lr{{\bf m}-{\bf k},{\bf n}-{\bf y}}-\frac{1}{2}\lr{{\bf m}-{\bf k}+{\bf y},{\bf k}+{\bf n}-{\bf y}}.$$
That is, we get that
\begin{flalign*}
\sigma_2&=q^{\frac{1}{2}\Lambda({\bf m}^*,{\bf n}^*)}\sum\limits_{\begin{smallmatrix}A,B,C,D,I,K,L,X,Y\\D\ncong N\end{smallmatrix}}q^{\lr{{\bf m}-{\bf k},{\bf n}-{\bf y}}-\frac{1}{2}\lr{{\bf m}-{\bf k}+{\bf y},{\bf k}+{\bf n}-{\bf y}}}\frac{a_CF_{CD}^NF_{AB}^{M}F_{IC}^{\tau B}}{q^{[M,N]^1}-1}\\&\quad\quad\quad\quad\quad\quad\times F_{KL}^AF_{XY}^DX^{-({\bf l}+{\bf y}+{\bf b})^*-^*({\bf k}+{\bf x}+{\bf c})}.
\end{flalign*}
In the above sum, suppose that $D\cong N$, then we have that $C=B=I=0$ and $A=M$. So we rewrite $\sigma_2$ as
\begin{flalign*}
\sigma_2&=q^{\frac{1}{2}\Lambda({\bf m}^*,{\bf n}^*)}\sum\limits_{\begin{smallmatrix}A,B,C,D,I,K,L,X,Y\end{smallmatrix}}q^{\lr{{\bf m}-{\bf k},{\bf n}-{\bf y}}-\frac{1}{2}\lr{{\bf m}-{\bf k}+{\bf y},{\bf k}+{\bf n}-{\bf y}}}\frac{a_CF_{CD}^NF_{AB}^{M}F_{IC}^{\tau B}}{q^{[M,N]^1}-1}\\&\quad\quad\quad\quad\quad\quad\times F_{KL}^AF_{XY}^DX^{-({\bf l}+{\bf y}+{\bf b})^*-^*({\bf k}+{\bf x}+{\bf c})}\\&
\quad-q^{\frac{1}{2}\Lambda({\bf m}^*,{\bf n}^*)}\sum\limits_{\begin{smallmatrix}K,L,X,Y\end{smallmatrix}}q^{\lr{{\bf l},{\bf x}}-\frac{1}{2}\lr{{\bf l}+{\bf y},{\bf k}+{\bf x}}}\frac{1}{q^{[M,N]^1}-1} F_{KL}^MF_{XY}^NX^{-({\bf l}+{\bf y})^*-^*({\bf k}+{\bf x})}.
\end{flalign*}
Using the associativity formulas of Hall algebras \begin{equation}\sum\limits_{A}F_{KL}^AF_{AB}^M=\sum\limits_{\tilde{A}}F_{K\tilde{A}}^MF_{LB}^{\tilde{A}}\quad\text{and}\quad
\sum\limits_{D}F_{CD}^NF_{XY}^D=\sum\limits_{\tilde{D}}F_{CX}^{\tilde{D}}F_{\tilde{D}Y}^{N},\end{equation}
we obtain that
\begin{flalign*}
\sigma_2&=q^{\frac{1}{2}\Lambda({\bf m}^*,{\bf n}^*)}\sum\limits_{\begin{smallmatrix}\tilde{A},B,C,\tilde{D},I,K,L,X,Y\end{smallmatrix}}q^{\lr{{\bf m}-{\bf k},{\bf n}-{\bf y}}-\frac{1}{2}\lr{{\bf m}-{\bf k}+{\bf y},{\bf k}+{\bf n}-{\bf y}}}\frac{a_CF_{CX}^{\tilde{D}}F_{LB}^{\tilde{A}}F_{IC}^{\tau B}}{q^{[M,N]^1}-1}\\&\quad\quad\quad\quad\quad\quad\times F_{K\tilde{A}}^MF_{\tilde{D}Y}^NX^{-(\tilde{{\bf a}}+{\bf y})^*-^*({\bf k}+\tilde{{\bf d}})}\\&
\quad-q^{\frac{1}{2}\Lambda({\bf m}^*,{\bf n}^*)}\sum\limits_{\begin{smallmatrix}K,L,X,Y\end{smallmatrix}}q^{\lr{{\bf l},{\bf x}}-\frac{1}{2}\lr{{\bf l}+{\bf y},{\bf k}+{\bf x}}}\frac{1}{q^{[M,N]^1}-1} F_{KL}^MF_{XY}^NX^{-({\bf l}+{\bf y})^*-^*({\bf k}+{\bf x})}.
\end{flalign*}
By \cite[Lemma 4.4]{Rupel2} and the Auslander-Reiten formula, we have that
\begin{flalign*}
\sum\limits_{I,L,X,B,C}a_CF_{CX}^{\tilde{D}}F_{LB}^{\tilde{A}}F_{IC}^{\tau B}=\sum\limits_{I,L,X}|_X\Hom_{\widetilde{\A}}(\tilde{D},\tau \tilde{A})_{\tau L\oplus I}|
=|\Hom_{\widetilde{\A}}(\tilde{D},\tau \tilde{A})|=|\Ext_{\widetilde{\A}}^1(\tilde{A},\tilde{D})|.
\end{flalign*}
Hence,
\begin{flalign*}
\sigma_2&=q^{\frac{1}{2}\Lambda({\bf m}^*,{\bf n}^*)}\sum\limits_{\begin{smallmatrix}\tilde{A},\tilde{D},K,Y\end{smallmatrix}}q^{\lr{\tilde{{\bf a}},\tilde{{\bf d}}}-\frac{1}{2}\lr{\tilde{{\bf a}}+{\bf y},{\bf k}+\tilde{{\bf d}}}}\frac{q^{[\tilde{A},\tilde{D}]^1}}{q^{[M,N]^1}-1} F_{K\tilde{A}}^MF_{\tilde{D}Y}^NX^{-(\tilde{{\bf a}}+{\bf y})^*-^*({\bf k}+\tilde{{\bf d}})}\\&
\quad-q^{\frac{1}{2}\Lambda({\bf m}^*,{\bf n}^*)}\sum\limits_{\begin{smallmatrix}K,L,X,Y\end{smallmatrix}}q^{\lr{{\bf l},{\bf x}}-\frac{1}{2}\lr{{\bf l}+{\bf y},{\bf k}+{\bf x}}}\frac{1}{q^{[M,N]^1}-1} F_{KL}^MF_{XY}^NX^{-({\bf l}+{\bf y})^*-^*({\bf k}+{\bf x})}.
\end{flalign*}
Unifying the indexes in the above two sums, we obtain that
\begin{flalign*}
\sigma_2=q^{\frac{1}{2}\Lambda({\bf m}^*,{\bf n}^*)}
\sum\limits_{A,B,C,D}\frac{q^{[B,C]^1}-1}{q^{[M,N]^1}-1}q^{\lr{{\bf b},{\bf c}}}q^{-\frac{1}{2}\lr{{\bf b}+{\bf d},{\bf a}+{\bf c}}}F_{AB}^MF_{CD}^NX^{-({\bf b}+{\bf d})^\ast-^\ast({\bf a}+{\bf c})}.
\end{flalign*}
Hence, $X_MX_N=\sigma_1+\sigma_2$, we complete the proof.
\end{proof}

Using Lemma \ref{zscx}, we reformulate Theorem \ref{ddlz} as the following
\begin{corollary}\label{maincor}
Let $M,N\in\widetilde{\A}$. Then we have the following equation in $\mathcal{T}_\Lambda:$
\begin{equation}\label{zuihou}
\begin{split}
&(q^{[M,N]^1}-1)X_M X_N=q^{\frac{1}{2}\Lambda({\bf m}^*,{\bf n}^*)}\sum\limits_{E\ncong M\oplus N}
|\Ext_{\widetilde{\A}}^1(M,N)_E|X_E+\\&q^{\frac{1}{2}\Lambda({\bf m}^*,{\bf n}^*)+\frac{1}{2}\lr{{\bf m},{\bf n}}}\sum\limits_{\begin{smallmatrix}D,A,I\\D\ncong N\end{smallmatrix}}q^{-\frac{1}{2}\lr{{\bf a},{\bf d}-{\bf i}}-\frac{1}{2}\Lambda({\bf a}^*,({\bf d}-{\bf i})^*)}|_D\Hom_{\widetilde{\A}}(N,\tau M)_{\tau A\oplus I}|X_{A} X_{D\oplus I[-1]},
\end{split}\end{equation}
where each $A$ has the same maximal projective direct summand as $M$.
\end{corollary}
\begin{remark}
%The formulas in Corollaries \ref{mcf1} and \ref{mutatione} can be viewed as the quantum cluster mutation multiplication formulas among the quantum cluster characters.
The formulas in Theorem \ref{dyggs} and Theorem \ref{ddlz} (Corollary \ref{maincor}) are different from the Hall multiplication formulas given in Corollary \ref{Hallcfgs}. They can be viewed as the quantum version of the cluster multiplication theorem in the classical cluster algebra, which were proved by Caldero-Keller \cite{CK2005} for finite type, Hubery \cite{Hubery1} for tame type and Xiao-Xu \cite{XX,Xu} for any acyclic quivers.
\end{remark}
The following corollary is a generalization of \cite[Theorem 4.5]{Rupel1}.
\begin{corollary}\label{mutatione}
Let $M,N\in\widetilde{\A}$ with a unique (up to scalar) non-trivial extension $E\in\Ext_{\widetilde{\A}}^1(M,N)$; in particular $\dim_{\End(M)}\Ext_{\widetilde{\A}}^1(M,N)=1$.
Let $\theta\in\Hom_{\widetilde{\A}}(N,\tau M)$ be the equivalent morphism with $A,D,I$ as before. Then we have the following equation
\begin{equation}
X_M X_N=q^{\frac{1}{2}\Lambda({\bf m}^*,{\bf n}^*)}X_E+q^{\frac{1}{2}\Lambda({\bf m}^*,{\bf n}^*)+\frac{1}{2}\lr{{\bf m},{\bf n}}-\frac{1}{2}\lr{{\bf a},{\bf d}-{\bf i}}-\frac{1}{2}\Lambda({\bf a}^*,({\bf d}-{\bf i})^*)}X_{A} X_{D\oplus I[-1]}.
\end{equation}
In particular, if $\Hom_{\widetilde{\A}}(A,I)=0=\Ext^1_{\widetilde{\A}}(A,D)$, then
\begin{equation}
X_M X_N=q^{\frac{1}{2}\Lambda({\bf m}^*,{\bf n}^*)}X_E+q^{\frac{1}{2}\Lambda({\bf m}^*,{\bf n}^*)+\frac{1}{2}\lr{{\bf m},{\bf n}}-\frac{1}{2}\lr{{\bf a},{\bf d}}} X_{A\oplus D\oplus I[-1]}.
\end{equation}
\end{corollary}

We now provide an example to interpret the formulas in our main Theorems \ref{dyggs} and \ref{ddlz}.
\begin{example}
Let $Q$ be the Kronecker quiver
$$
\xymatrix {1  \ar @<2pt>[r] \ar @<-2pt>[r]&  2}.
$$
Then
$$R^{'}=\left(\begin{array}{cc} 0 & 2\\
0& 0
\end{array}\right),\
R=\left(\begin{array}{cc} 0 & 0\\
2& 0
\end{array}\right),\
I-R^{'}=\left(\begin{array}{cc} 1 & -2\\
0& 1
\end{array}\right),\
B=\left(\begin{array}{cc} 0 & 2\\
-2& 0
\end{array}\right).$$

We take
$\Lambda=\left(\begin{array}{cc} 0& 1\\
-1 & 0
\end{array}\right)$, and it satisfies that
\begin{align*}\label{eq:simply_laced_compatible}
\Lambda(-B)=\left(\begin{array}{cc} 2 & 0\\
0 & 2 \end{array}\right).
\end{align*}
Hence, in what follows, we consider the representations of $Q$ over the field $k=\mathbb{F}_{q^2}$.
For each $i=1,2$, we denote by $S_i$ the simple module corresponding to the vertex $i$, and by $P_i$ and $I_i$ the projective cover and the injective envelope of $S_i$, respectively.

For any $p$ in the projective line $\mathbb{P}^{1}(k)$ of degree 1,
we denote by $R_p(l)$ the regular indecomposable module in the
homogeneous tube with the dimension vector $(l, l)$ for $l \in \mathbb{Z}_{\geq 1}$.
By the definition of the quantum cluster character $X_M$, we have that $$X_{R_{p}(1)}=X^{e_1-e_2}+X^{-e_1-e_2}+X^{-e_1+e_2},\ \ \text{for any} \ p\in \mathbb{P}^{1}(k),$$
where $e_1,e_2$ is the 2-dimensional unit basis vector.
So we denote $X_{R_{p}(1)}$ by  $X_\delta$.

By using the definition of the quantum cluster character and direct calculations, we obtain that (see also \cite{dx1})
\begin{equation}\label{jianyan2}X_{S_1} X_{S_2}=q^{\frac{1}{2}}X_\delta+q^{-\frac{3}{2}}X_{(I_1\oplus I_2)[-1]}\end{equation}
and
\begin{equation}\label{jianyan1}
X_{P_2[1]}X_{P_1}=q^{\frac{1}{2}}X_\delta+q^{-\frac{3}{2}}X_{S_2\oplus I_1[-1]}.\end{equation}

On the other hand,   we have the following exact sequence
\[
  0 \longrightarrow S_2 \longrightarrow  R_p(1)  \longrightarrow  S_1   \longrightarrow  0
 \] for each $\ p\in \mathbb{P}^{1}(k)$, and
$\sum\limits_{p\in \mathbb{P}^{1}(k)}|\Ext^1_{\mathfrak{S}}(S_1,S_2)_{R_p(1)}|=(q^2)^2-1=q^4-1$. Thus,
$$q^{\frac{1}{2}\Lambda({e_1}^*,{e_2}^*)}\sum\limits_{E\ncong S_1\oplus S_2}
|\Ext^1_{\mathfrak{S}}(S_1,S_2)_E|X_E=q^{\frac{1}{2}}(q^4-1)X_\delta.$$

It is easy to see that for any nonzero morphism $f: S_2 \rightarrow \tau S_1$ we have that $\Ker f=0$ and $\Coker f\cong I_1\oplus I_2$.
Thus, the second term on the righthand side of the equation (\ref{zuihou}) is equal to the following
\begin{flalign*}&q^{\frac{1}{2}\Lambda({e_1}^*,{e_2}^*)+\frac{1}{2}\lr{{e_1},{e_2}}}|_0\Hom_{\mathfrak{S}}(S_2,\tau S_1)_{I_1 \oplus I _2}|X_{(I_1\oplus I_2)[-1]}\\
&=q^{\frac{1}{2}+\frac{1}{2}\times(-4)}((q^2)^2-1)X_{(I_1\oplus I_2)[-1]}\\
&=q^{-\frac{3}{2}}(q^4-1)X_{(I_1\oplus I_2)[-1]}.\end{flalign*}
%We can also deduce the following exact sequence
%\[
%  0  \longrightarrow  S_2   \longrightarrow \tau S_1   \longrightarrow I_1 \oplus I _2\longrightarrow 0.
% \]

%Note that,  we have $[S_1,S_2]^{1}=4$ and $|_0\Hom(S_2,\tau S_1)_{I_1 \oplus I _2}|=q^4-1$,  and thus
%\begin{equation}
%\begin{split}
%&q^{\frac{1}{2}\Lambda({\bf s_1}^*,{\bf s_2}^*)+\frac{1}{2}\lr{{\bf s_1},{\bf s_2}}}\sum\limits_{\begin{smallmatrix}D,A,I\\D\ncong N\end{smallmatrix}}q^{-\frac{1}{2}\lr{{\bf a},{\bf d}-{\bf i}}-\frac{1}{2}\Lambda({\bf a}^*,({\bf d}-{\bf i})^*)}|_D\Hom(S_2,\tau S_1)_{\tau A\oplus I}|X_{A} X_{D\oplus I[-1]}\\&=q^{\frac{1}{2}-2}|_0\Hom(S_2,\tau S_1)_{I_1 \oplus I _2}|X_{I_1[-1]\oplus I_2[-1]}\\&=q^{-\frac{3}{2}}(q^4-1)X_{I_1[-1]\oplus I_2[-1]}.
%\end{split}\end{equation}
Hence, by Corollary \ref{maincor}, we have that $$(q^4-1)X_{S_1} X_{S_2}=q^{\frac{1}{2}}(q^4-1)X_\delta+q^{-\frac{3}{2}}(q^4-1)X_{(I_1\oplus I_2)[-1]},$$
which is the same as the equation (\ref{jianyan2}).

Similarly, we illustrate the first formula in Theorem \ref{dyggs} for $M=P_1$ and $P=P_2$ as the following
$$(q^4-1)X_{P_2[1]}X_{P_1}=q^{\frac{1}{2}}((q^4-1)X_\delta+q^{-2}(q^4-1)X_{S_2\oplus I_1[-1]}),$$
which is the same as the equation (\ref{jianyan1}).

\end{example}

\section*{Acknowledgments}
We would like to thank Professor Fan Xu for his helpful and valuable suggestions and comments on this project, and the consultations with him played a crucial role in the final stage of this manuscript.
%The authors are grateful to the anonymous referees for their valuable suggestions and comments,
%and supported partially by the National Natural Science Foundation of China (No.s 11771217, 11471177, 11801273), Natural Science Foundation of Jiangsu Province of China (No.BK20180722)
%and Natural Science Foundation of Jiangsu Higher Education Institutions of China  (No.18KJB110017).

\end{document}